\documentclass[onealgnum]{siamart220329}
\usepackage{amsfonts}
\usepackage{enumerate}
\usepackage{multirow,booktabs,siunitx}
\usepackage{array} 
\usepackage{algorithmic}
\usepackage{graphicx,epstopdf}
\usepackage{subcaption}
\usepackage[percent]{overpic}
\usepackage{tikz}
\usepackage{adjustbox}
\usepackage{todonotes}
\usepackage[hyperpageref]{backref}
\renewcommand*{\backref}[1]{}
\renewcommand*{\backrefalt}[4]{%
    \ifcase #1 (Not cited.)%
    \or        (Cited on page~#2.)%
    \else      (Cited on pages~#2.)%
    \fi}
\usepackage{verbatim}

\usepackage[numbers,sort&compress]{natbib}

\newtheorem{thm}{\bf Theorem}

\newtheorem{rem}{ Remark}[section]

\newcommand{\diff}{\mathop{}\!\mathrm{d}}
\newcommand{\bfn}{\mathbf{n}}

\title{
A Fast Minimization Algorithm for the Euler Elastica Model Based on a Bilinear Decomposition
\thanks{
\funding{The work of the first and last authors is partially supported by the NSFC grants (Nos. 12271404,  11871372,  11501413, and 12301545) and PHD Program 52XB2013 of Tianjin Normal University. The work of Xue-Cheng Tai is partially supported by NSFC/RGC grant N-HKBU214-19 and NORCE Kompetanseoppbygging program.}
}
}

\author{Zhifang Liu\thanks{School of Mathematical Sciences, Tianjin Normal University, Tianjin 300387, China (\email{matlzhf@tjnu.edu.cn}, \email{bcsun\_tj@163.com}).}
\and Baochen Sun\footnotemark[2]
\and Xue-Cheng Tai\thanks{Norwegian Research Center (NORCE), Nygårdsgaten 112, 5008 Bergen, Norway  (\email{xtai@norceresearch.no}).}
\and Qi Wang\thanks{Department of Mathematics,  University of South Carolina, Columbia, SC 29208, USA (\email{qwang@math.sc.edu}).}
\and Huibin Chang\thanks{Corresponding author. School of Mathematical Sciences, Tianjin Normal University, Tianjin 300387, China (\email{changhuibin@tjnu.edu.cn}).}
}

\headers{A Fast Minimization Algorithm for the Euler Elastica Model}{Zhifang Liu, Baochen Sun, Xue-Cheng Tai, Qi Wang, and Huibin Chang}

\begin{document}

\maketitle

\begin{abstract}
The Euler Elastica (EE) model with surface curvature can generate artifact-free results compared with the traditional total variation regularization model in image processing. However, strong nonlinearity and singularity due to the curvature term in the EE model pose a great challenge for one to design fast and stable algorithms for the EE model. In this paper,  we propose a new, fast, hybrid alternating minimization (HALM) algorithm for the EE model based on a  bilinear decomposition of the gradient of the underlying image and prove the global convergence of the minimizing sequence generated by the algorithm under mild conditions. The HALM algorithm comprises three sub-minimization problems and each is either solved in the closed form or approximated by fast solvers making the new algorithm highly accurate and efficient. 
We also discuss the extension of the HALM strategy to deal with general curvature-based variational models, especially with a Lipschitz smooth functional of the curvature. A host of numerical experiments are conducted to show that the new algorithm produces good results with much-improved efficiency compared to other state-of-the-art algorithms for the EE model. As one of the benchmarks, we show that the average running time of the HALM algorithm is at most one-quarter of that of the fast operator-splitting-based Deng-Glowinski-Tai algorithm.
\end{abstract}

\begin{keywords}
    Euler Elastica; Hybrid alternating minimization algorithm; Regularization; Bilinear decomposition of the gradient; Curvature-based variational models
\end{keywords}

\begin{MSCcodes}
 46N10; 47N10; 65K10; 68U10
\end{MSCcodes}


\section{Introduction}

We consider  solving the Euler Elastica (EE) model numerically  for image $u$ defined in $\mathcal{C}^2(\Omega;\mathbb{R})$, $\Omega\subset\mathbb{R}^2$, i.e.,
\begin{align}
\min_{u}\int_{\Omega} \Bigg(a + b\bigg(\operatorname{div} \frac{\nabla u}{|\nabla u|}\bigg)^2 \Bigg)|\nabla u|\diff x\diff y + \frac{1}{2} \int_{\Omega} (u-f)^2\diff x \diff y,
\label{originalEEM}
\end{align}
where $a$ and $b$ are two  positive parameters, and $f$ is a given noisy image. The first term, known as the EE energy \cite{Mumford1994Elastica,Masnou1998Level,Chan2002Euler}, regularizes the lengths and curvatures of the level curves in the image and thus warrants strong continuity of the edges of underlying image $u$. The second term denotes the fidelity term that measures the difference between given noisy image $f$ and its denoised approximation $u$. Such a model has been used for image restoration tasks \cite{Tai2011AugmentedLagrangian,Bredies2015convex,2019DengOS}.
However, the energy functional in  \eqref{originalEEM} has a strongly nonlinear term as well as a singularity at places where the gradient vanishes. This poses a great challenge for one to optimize it.

Traditionally, algorithms for the EE model have been designed directly based on the gradient flow method \cite{Chan2002Euler,RN2Split_Model,Ringholm2018Variational}.
The Chan-Kang-Shen (CKS) method \cite{Chan2002Euler} used a gradient descent scheme for the fourth-order nonlinear Euler-Lagrange equation of the EE model.  In order to meet the Courant-Friedrichs-Lewy condition, a small step size needs to be chosen, leading to slow convergence of the CKS method. An accelerated version of the CSK method was proposed by Yashtini and Kang  \cite{RN2Split_Model} via the Nesterov's technique \cite{Nesterov1983method}.
Ringh{\o}lm, Lazi\'{c}, and Sch\"{o}nlieb introduced the discrete gradient scheme to the gradient flow of the smooth variant of the EE model based on the Itoh–Abe discrete gradient scheme \cite{Ringholm2018Variational}. Wang et al. \cite{wang2022efficient} proposed efficient scalar auxiliary variable algorithms for more general curvature minimization problems.
In order to deal with nonconvexity of \eqref{originalEEM},  convex approximations \cite{Bredies2015convex} or relaxation methods \cite{Chambolle2019Total} were used to reformulate the EE energy in higher dimensional spaces.
Bredies, Pock, and Wirth in \cite{Bredies2015convex} proposed a convex, lower semi-continuous, coercive approximation of the EE energy by functional lifting of the image gradient, and showed some promising results using a tailored discretization of measures. In the work by Chambolle and Pock \cite{Chambolle2019Total}, the EE energy was represented through a convex functional defined on divergence-free vector fields, and successfully applied to diverse shape and image processing tasks utilizing a staggered grid discretization baseed on an averaged Raviart-Thomas finite element approximation.

The variable splitting method  was developed to address issues of strong nonlinearity, nonsmoothness, and singularity of \eqref{originalEEM} in \cite{Tai2011AugmentedLagrangian,Duan2012Fast,Zhu2013Image,Zhang2017Fast,RN2Split_Model,2019DengOS,Liu2019Proximal,Liu2021variable,He2021Penalty,Zhang2022Image}.
By introducing auxiliary variables, Tai, Hahn, and Chung in \cite{Tai2011AugmentedLagrangian} proposed a sophisticated relaxation of the EE model and reformulated the energy minimization problem into an equivalent constrained optimization problem, which was then solved by an augmented Lagrangian method \cite{glowinski1989augmented,wu2010augmented}.
Inspired by this work, there have been extensive studies of variable splitting and augmented Lagrangian method for the EE model.
Duan, Wang, and Hahn \cite{Duan2012Fast} introduced one more auxiliary variable for the mean curvature and presented a new augmented Lagrangian method.
Zhang and Chen in \cite{Zhang2016New} proposed an augmented Lagrangian primal-dual algorithm for the EE model.
Yashtini and Kang in \cite{RN2Split_Model} relaxed the normal vector in the curvature term of the EE model and developed the relaxed normal two-split (RN2Split) method.
 Zhang et al. in \cite{Zhang2017Fast} and Liu et al. in \cite{Liu2019Proximal} improved the classical augmented Lagrangian method for the EE model via a linearization technique and the proximal method, respectively.
Recently, based on the Lie scheme, Deng, Glowinski, and Tai \cite{2019DengOS} proposed a stable and fast operator-splitting algorithm dubbed the DGT algorithm.  
Liu and his collaborators extended the DGT algorithm to a color EE model \cite{Liu2021Color} and the Gaussian curvature regularization model \cite{Liu2022Gaussian}.
He, Wang, and Chen in \cite{He2021Penalty}, adopted a smoothed constrained relaxation model of \eqref{originalEEM} and proposed a convergent penalty relaxation algorithm (dubbed the HWC algorithm) for the discrete EE model.

Although most variable-splitting methods  can solve the EE model efficiently, their theoretical convergence is difficult to prove due to the complex relations  among the coupled auxiliary variables. Especially in \cite{2019DengOS},  
the  objective functionals based on the Marchuk--Yanenko  scheme \cite{Glowinski2016} for the DGT algorithm are nonsmooth and nonconvex such that it is still unclear how to prove the local convergence to the  stationary points of the EE model.  Numerical oscillations observed from the relative error (See Figures 4, 5 in \cite{2019DengOS}) may affect the convergence speed. By introducing an additional auxiliary variable for the Hessian matrix,  an operator-splitting algorithm was developed for a more complicated Gaussian-curvature regularized model recently \cite{Liu2022Gaussian}, which was also  based on the Marchuk--Yanenko scheme. However, a rigorous proof of convergence is still elusive and some numerical oscillations remain. 
We note that the HWC algorithm guarantees convergence but could inflict a high computational cost. It used the lagged diffusivity fixed point iteration and the scheme for the subproblems has to update the coefficient matrices per inner loop. 

This paper proposes a new convergent variable-splitting algorithm for the EE model \eqref{originalEEM} based on a simple yet powerful regularizing, bilinear decomposition. Notice that gradient $\nabla u$ of the underlying image $u$ can be expressed as follows 
\[
\nabla u=|\nabla u| \frac{\nabla u}{|\nabla u|}.
\]
Denoting $q:=|\nabla u|$ and  $\vec{n}:=\frac{\nabla u}{|\nabla u|}$,   the gradient is equivalent  to the following  bilinear decomposition  
\[
\nabla u=q \vec{n}\qquad \mbox{with~ } q\geq 0, \ \ |\vec{n}|=1,
\] 
with two additional constraints.
Using the decomposition, the original EE model can be reformulated as a smooth optimization problem with a non-negative and a unit-length constraint.  Namely, the EE energy is expressed by a smooth functional of normal vector $\vec{n}$ and magnitude $q$. The singularity in the energy disappears and the strong nonlinearity is also mitigated through the bilinear decomposition. When applying it to image denoising, we further simplify the problem by penalizing the bilinear relation, i.e. adding term $\alpha \|\nabla u-q \vec{n}\|^2$ to the energy with a large $\alpha>0$ instead of imposing the bilinear decomposition through exact penalization or saddle-point problem traditionally used in constrained optimization. This simplification with a fixed $\alpha>0$ significantly reduces the computational cost.

The reformulation of the EE energy based on a bilinear decomposition of the gradient, relaxation in the constraints through penalties, and the variable splitting technique effectively eliminates the singularity, mitigates the nonlinearity in the original EE model, and improves the regularity of the reformulated energy (objective functional).
We then use the finite difference method to discretize the EE energy in space to arrive at a  reformulated, discretized EE model. We develop the fast, steadily convergent, hybrid alternating minimization algorithm (HALM) for the discrete EE model. Each subproblem in the new optimization algorithm is solved either in a closed form or approximated efficiently by fast solvers. Under mild conditions, the algorithm is shown to produce a globally convergent minimizing sequence. 
We discuss the same strategy extended to a list of   general curvature-based models, where the corresponding convergence is guaranteed with an $L-$Lipschitz smooth curvature term.

Finally, we conduct a host of experiments to show the effectiveness of the algorithms for image denoising. Compared to the other state-of-the-art algorithms, including the DGT and HWC algorithms, the new algorithm converges much faster for both the EE model and general curvature-based model. In benchmark examples, we show that this algorithm requires only one-quarter of the running time to reach the same given tolerance compared to the other fast operator-splitting-based DGT algorithms. In addition, the algorithm is easy to implement numerically since no staggered grid discretization is needed, and only two algorithmic parameters need to be tuned.  This essentially three-step  algorithm of low computational complexity and costs 
has a great potential to be applied to other curvature-based models.

The rest of the paper is organized as follows. Section \ref{sec-model} presents the reformulated model and its discretization. Section \ref{sec-alg} provides the proposed algorithm for the reformulated discrete model with a rigorous convergence proof. It is extended to the generalized curvature-based model in section \ref{sec-extension}. Numerical experiments are conducted in section \ref{sec-num} to validate the proposed algorithm and compare the HALM algorithm with the other state-of-the-art ones.  Finally, we give conclusions
and discuss our future work in section \ref{sec-conclusion}. 

\section{Model reformulation and discretization}\label{sec-model}

In this section, we reformulate the EE model \eqref{originalEEM} with a bilinear decomposition and then present its discrete form.

\subsection{Reformulation of the EE model}

A typical difficulty in dealing with the energy in \eqref{originalEEM} is its weak singularity, i.e. the normal field $\vec{n} = \frac{\nabla u}{|\nabla u|}$ makes no sense at $\{x: |\nabla u(x)| = 0\}$ and may cause numerical instability when the denominator is small.  
By introducing auxiliary variable $q$, one can express the gradient $\nabla u$ by the bilinear decomposition:
\begin{equation}\label{eq:ref2}
\nabla u = q \vec{n},~|\vec{n}| = 1,~q\geq 0, 
\end{equation}
and reformulate EE model \eqref{originalEEM} into the following equivalent minimization problem
\begin{equation}
\begin{aligned}
&\min_{u,q,\vec{n}}  \int_{\Omega} \big(a + b (\operatorname{div} \vec{n})^2\big)q\diff x\diff y + \frac{1}{2}\int_{\Omega} (u-f)^2\diff x\diff y\\
&\;\; \mbox{s.t.}\;  \nabla u = q \vec{n},~|\vec{n}| = 1,~q\geq 0.
\end{aligned}
\label{reformulateEEM}
\end{equation}
One may notice that not only does the singularity of the objective functional disappears, but the strong nonlinearity is also mitigated.

We remark that one can extend the bilinear decomposition technique to general curvature-based objective functional \cite{Duan_Min_discrete_total_curva}
\begin{equation}\label{general_model-1}
\min_{u} \int_{\Omega} \phi( \operatorname{div} (\tfrac{\nabla u}{|\nabla u|})|\nabla u| \diff x \diff y + \frac{1}{2}\int_{\Omega} (u-f)^2\diff x\diff y,
\end{equation}
with function $\phi$ (Specific forms will be given in section \ref{sec-extension}).
Based on the bilinear decomposition given in \eqref{eq:ref2}, one immediately turns the above minimization problem into an equivalent one as   follows:
\begin{equation}\label{general_model_bilinear-1}
\begin{aligned}
&\min_{u,\vec{n}}  \int_{\Omega} \phi(\operatorname{div} \vec{n}) q \diff x\diff y + \frac{1}{2}\int_{\Omega} (u-f)^2\diff x\diff y,\\
&\;\; \mathrm{s.t.}\;  \nabla u = q \vec{n},~|\vec{n}| = 1,~q\geq 0.
\end{aligned}
\end{equation}

\subsection{Discretization and penalty relaxation}

We consider a meshed rectangle domain: $\Omega  = [x_1, x_{N}]\times [y_1, y_{N}]$ with mesh sizes $\delta x = (x_{N}-x_1)/(N-1)=1$ and $\delta y = (y_{N}- y_1)/(N-1)=1$. The corresponding discrete image domain $\Omega_d$ is defined by
\[
\Omega_d = \left\{\left(x_i, y_j\right):x_i = x_1 + (i-1)\delta x, y_j = y_1 + (j-1)\delta y,~i,j = 1, \ldots, N \right\}.
\]
We represent image $u$ as a matrix with entries $u(x_i,y_j) (i,j = 1, \ldots, N)$ defined in $\Omega_d$.
For convenience, we rearrange this matrix into a vector $\mathbf{u} = (u_1,u_2,\ldots,u_{N^2})^{\mathrm{T}}\in\mathbb{R}^{N^2}$ by the lexicographical column ordering.
Given $\bfn_1 = (n_{1,1},n_{1,2},\ldots,n_{1,N^2})^{\mathrm{T}}$ and  $\bfn_2 = (n_{2,1},n_{2,2},\ldots,n_{2,N^2})^{\mathrm{T}}$, belonging to $\mathbb{R}^{N^2}$,
we denote
$\bfn :=\begin{bmatrix}
\bfn_1 \\
\bfn_2
\end{bmatrix} \in \mathbb{R}^{2 N^2}$.
Analogously, vectors $\mathbf{q}\in\mathbb{R}^{N^2}$ and $\mathbf{f}\in\mathbb{R}^{N^2}$ are defined.

We introduce two $N \times N$ matrices $\mathbf{D}_1$ and $\mathbf{D}_2$   as follows
\begin{equation*}
\mathbf{D}_1=
\begin{bmatrix}
-1  & 1 & & \\ & -1 & \ddots & \\ & & \ddots & 1 \\ 1 & & & -1
\end{bmatrix},~
\mathbf{D}_2= -\mathbf{D}_1^{\mathrm{T}},
\end{equation*}
and four $N^2 \times N^2$ matrices $\mathbf{D}_x^{+}$, $\mathbf{D}_y^{+}$, $\mathbf{D}_x^{+}$ and $\mathbf{D}_y^{+}$:
\[ \mathbf{D}_x^{+} = \mathbf{I} \otimes \mathbf{D}_1, ~
\mathbf{D}_y^{+} = \mathbf{D}_1 \otimes \mathbf{I}, ~
\mathbf{D}_x^{-} = \mathbf{I} \otimes \mathbf{D}_2 , ~
\mathbf{D}_y^{-} = \mathbf{D}_2 \otimes \mathbf{I},\]
where $\mathbf{I} \in \mathbb{R}^{N\times N}$ is the identity matrix, and the symbol $\otimes $ denotes the Kronecker product of two matrices.
Readily, one can see that
\begin{equation}
\mathbf{D}_x^{-} = - (\mathbf{D}_x^{+})^{\mathrm{T}},\quad \mathbf{D}_y^{-} = - (\mathbf{D}_y^{+})^{\mathrm{T}}. \label{forwardBackwardRelation}
\end{equation}
Assuming the periodic boundary condition for $u$, we have
\[
u(x_1, y_j) = u(x_{N+1}, y_j), \; u(x_i, y_1) = u(x_i, y_{N+1}),\ \ i,j = 1, \ldots, N.
\]
Gradient $\nabla $ and divergence $\operatorname{div} $  are discretized as follows
\[
\nabla\mathbf{u}=
\begin{bmatrix}
\mathbf{D}_x^{+}\mathbf{u} \\
\mathbf{D}_y^{+}\mathbf{u}
\end{bmatrix}
\in \mathbb{R}^{2N^2}, ~
\operatorname{div}\bfn=\mathbf{D}_x^{-}\bfn_1 +\mathbf{D}_y^{-}\bfn_2 \in \mathbb{R}^{N^2}.
\]

\begin{rem}
In variational image processing, the periodic boundary condition is commonly used for historical reasons as well as its benefits in enabling the use of fast discrete Fourier transforms to efficiently solve the linear systems associated with the $\mathbf{u}$-subproblem as in the DGT and HWC algorithms. The Neumann boundary condition could also be a natural choice for boundary conditions. With Neumann boundary conditions, the $\mathbf{u}$-subproblem can be solved via efficient methods like conjugate gradient (CG) descent and sparse Cholesky factorization.
In this work, we will primarily use the periodic boundary condition by default. However, we will also explain how to handle Neumann boundary conditions where appropriate. 
\end{rem}

Denote the $i$th row of  the matrix $\mathbf{D}_x^{-}$ (or $\mathbf{D}_y^{-}$) as $\mathbf{D}_{x,i}^{-}$ (or $\mathbf{D}_{y,i}^{-}$).
Using the notations above, we obtain the discrete form of \eqref{reformulateEEM}:
\begin{equation}
\begin{aligned}
& \min_{\mathbf{u},\bfn,\mathbf{q}}  E(\mathbf{u},\bfn,\mathbf{q}) = \sum\nolimits_{i=1}^{N^2} (a + b |\mathbf{D}_{x,i}^{-}\bfn_1+\mathbf{D}_{y,i}^{-}\bfn_2|^2)q_i + \frac{1}{2} \left\| \mathbf{u}-\mathbf{f}\right\|^2\\
&\quad \; \text{s.t.} \ \ \mathbf{D}_x^{+} \mathbf{u} =\mathbf{q}\odot \bfn_1,~\mathbf{D}_y^{+} \mathbf{u} =\mathbf{q}\odot \bfn_2,\\
&\qquad\quad\  \bfn_1\odot\bfn_1 + \bfn_2\odot\bfn_2 = \mathbf{1},\\
&\qquad\quad \ \mathbf{q}\geq \mathbf{0},
\end{aligned}
\label{reformulateEEM_discretization}
\end{equation}
where $\mathbf 1$ represents a vector whose elements are all ones, and $\|\cdot\|$ represents the standard $\ell^2$ norm, $\odot$ denotes component-wise product, and $\mathbf{q}= (q_1,q_2,\ldots,q_{N^2})^{\mathrm{T}} \geq\mathbf{0}$ denotes the non-negativity of each element.
There exists a large family of efficient first-order operator-splitting algorithms \cite{Glowinski2016} to solve discrete model \eqref{reformulateEEM_discretization} with the bilinear constraint.

Instead of strictly enforcing the constraints, we penalize the bilinear constraint in the objective functional and optimize the augmented objective functional using the alternating minimization method.
Specifically, we approximate the constrained minimization problem in \eqref{reformulateEEM_discretization} by an unconstrained optimization problem given below
\begin{equation}
\min_{\mathbf{u},\bfn,\mathbf{q}}  E_{\alpha,\mathbb{I}}(\mathbf{u},\bfn,\mathbf{q}),
\end{equation}
where
\begin{equation}
\begin{aligned}
 E_{\alpha,\mathbb{I}}(\mathbf{u},\bfn,\mathbf{q}) & =\sum\nolimits_{i=1}^{N^2} (a + b |\mathbf{D}_{x,i}^{-}\bfn_1+\mathbf{D}_{y,i}^{-}\bfn_2|^2)q_i + \frac{1}{2} \left\| \mathbf{u}-\mathbf{f}\right\|^2 \\
& \quad + \frac{\alpha}{2} \left(\|\mathbf{D}_x^{+} \mathbf{u} -\mathbf{q}\odot \bfn_1\|^2 + \| \mathbf{D}_y^{+} \mathbf{u} -\mathbf{q}\odot \bfn_2\|^2 \right)\\
& \quad + \mathbb{I}_{\mathcal{S}}(\bfn) + \mathbb{I}_{\mathbb{R}_{+}^{N^2}}(\mathbf{q}),
\label{DisrelaxationREEM-1}
\end{aligned}
\end{equation}
\[
\mathcal{S}=\left\{\bfn =\begin{bmatrix}
\bfn_1 \\
\bfn_2
\end{bmatrix}:\mathbf{n}_1, \mathbf{n}_2 \in \mathbb{R}^{N^2}, \mathbf{n}_1\odot\mathbf{n}_1 + \mathbf{n}_2\odot\mathbf{n}_2 = \mathbf{1}\right\},
\]
$\alpha$ is a positive parameter, $\mathbb{R}_{+}^{N^2}=\{\mathbf{p} \in \mathbb{R}^{N^2} :\mathbf{p} = (p_1,p_2,\ldots,p_{N^2})^{\mathrm{T}},\; p_i\geq 0, \; \forall i=1,2,\cdots,N^2\}$,
and the indicator function $\mathbb I$ is defined by
\[\mathbb{I}_{\mathcal{A}}(s)=
\left\{
\begin{aligned}
&0,& \mbox{if~} s\in \mathcal A;\\
&+\infty,&\mbox{otherwise.}
\end{aligned}
\right.
\]

\begin{thm}
Model \eqref{DisrelaxationREEM-1} has at least one minimizer.
\end{thm}
The proof of the theorem follows that of theorem 1.4.1 in \cite{MinimizerExistenceTheorem} and is thus omitted.
We remark that  penalized or relaxation model \eqref{DisrelaxationREEM-1} is equivalent to the original model \eqref{reformulateEEM_discretization} only when $\alpha \to \infty$ \cite{Nocedal2006Numerical}.  In practical use,  model \eqref{DisrelaxationREEM-1} with a sufficiently large parameter $\alpha$ could generate reasonable denoising results, which is also validated  numerically.


\section{Hybrid alternating minimization algorithm and its convergence}\label{sec-alg}

In this section, we present an efficient hybrid alternating minimization method for problem \eqref{DisrelaxationREEM-1}  and prove the global convergence of the iterative sequence.

We introduce a smooth function
\begin{align}
E_{\alpha}(\mathbf{u},\bfn,\mathbf{q})  &= \sum\nolimits_{i=1}^{N^2} (a + b |\mathbf{D}_{x,i}^{-}\bfn_1+\mathbf{D}_{y,i}^{-}\bfn_2|^2)q_i + \frac{1}{2} \left\| \mathbf{u}-\mathbf{f}\right\|^2 \label{eq-auxiliary-function}\\
&\quad + \frac{\alpha}{2} \big(\|\mathbf{D}_x^{+} \mathbf{u} -\mathbf{q}\odot \bfn_1\|^2 + \| \mathbf{D}_y^{+} \mathbf{u} -\mathbf{q}\odot \bfn_2\|^2 \big)\nonumber
\end{align}
to rewrite the objective function in \eqref{DisrelaxationREEM-1} as
\begin{equation*}
E_{\alpha,\mathbb{I}}(\mathbf{u},\bfn,\mathbf{q})= E_{\alpha}(\mathbf{u},\bfn,\mathbf{q})  + \mathbb{I}_{\mathcal{S}}(\bfn) + \mathbb{I}_{\mathbb{R}_{+}^{N^2}}(\mathbf{q}).
\end{equation*}
Note that although $E_{\alpha,\mathbb{I}}$ is nonconvex, it is strongly convex with respect to $\mathbf{u}$ and $\mathbf{q}$, respectively. The projection onto the sphere $\mathcal{S}$ is well-defined.
Next, based on the alternating minimization strategy, we present an iterative algorithm to minimize $E_{\alpha,\mathbb{I}}$.

\subsection{Hybrid Alternating Minimization Method}

We consider solving problem \eqref{DisrelaxationREEM-1} by  minimizing objective function \eqref{DisrelaxationREEM-1}  with respect to variables $\mathbf{u}$, $\bfn$, and $\mathbf{q}$ alternately.
Supposing that we have  $(\mathbf{u}^k,\mathbf{n}^k,\mathbf{q}^k)$ at the current $k$-iteration,
the proposed  {\bf H}ybrid  {\bf Al}ternating {\bf M}inimization method (HALM) for solving the bilinear decomposition-based EE model \eqref{DisrelaxationREEM-1} updates the variables  $(\mathbf{u}^{k+1},\mathbf{n}^{k+1},\mathbf{q}^{k+1})$ alternatively as follows
\begin{equation*}
\left\{
\begin{aligned}
\mathbf{u}^{k+1}& =\arg\min_{\mathbf{u}}E_{\alpha,\mathbb{I}}(\mathbf{u},\mathbf{n}^k,\mathbf{q}^k),\vspace{2ex}\\
\mathbf{n}^{k+1}& =\arg\min_{\mathbf{n}} \mathbb{I}_{\mathcal{S}}(\mathbf{n}) +
\frac{1}{2\tau_k} \| \mathbf{n} - \left( \mathbf{n}^k - \tau_k \nabla_{\mathbf{n}} E_{\alpha}(\mathbf{u}^{k+1},\mathbf{n}^k,\mathbf{q}^k) \right) \|^2,\vspace{2ex}\\
\mathbf{q}^{k+1}&=\arg\min_{\mathbf{q}}E_{\alpha,\mathbb{I}}(\mathbf{u}^{k+1},\mathbf{n}^{k+1},\mathbf{q}),
\end{aligned}
\right.
\end{equation*}
where $\tau_k$ is a positive parameter, $ E_{\alpha}$ is defined in \eqref{eq-auxiliary-function}, and
\[
\nabla_{\mathbf{n}} E_{\alpha}(\mathbf{u},\mathbf{n},\mathbf{q}):=
\begin{bmatrix}
    \nabla_{\bfn_1} E_{\alpha}(\mathbf{u},\mathbf{n},\mathbf{q})\\
    \nabla_{\bfn_2} E_{\alpha}(\mathbf{u},\mathbf{n},\mathbf{q})
\end{bmatrix}
\in \mathbb{R}^{2N^2}.
\]
Here, we adopt the forward-backward splitting approach to approximately solve the $\mathbf{n}-$subproblem as
\[\mathbf n^{k+1}\approx\min_{\mathbf n} E_{\alpha,\mathbb{I}}(\mathbf{u}^{k+1},\bfn,\mathbf{q}^k).\]
We will see that each subproblem in the HALM algorithm can be efficiently solved numerically or has a closed-form solution.

The $\mathbf{u}$-subproblem is a standard unconstrained quadratic optimization problem
\begin{align}
\mathbf{u}^{k+1} &= \arg\min_\mathbf{u}~ \frac{1}{2}\|\mathbf{u} - \mathbf{f}\|^2 + \frac{\alpha}{2} \big(\|\mathbf{D}_x^{+} \mathbf{u} -\mathbf{q}^k\odot \mathbf{n}_1^k\|^2\label{u_update_opt} \\
&\qquad\qquad\qquad\qquad\qquad\qquad+ \| \mathbf{D}_y^{+} \mathbf{u} -\mathbf{q}^k\odot \mathbf{n}_2^k\|^2 \big).\nonumber
\end{align}
The unique minimum solution is given by 
\begin{align}
 \mathbf{u}^{k+1} &=\big(\mathbf{I} \otimes \mathbf{I}  + \alpha( \mathbf{D}_x^{-}\mathbf{D}_x^{+}+\mathbf{D}_y^{-}\mathbf{D}_y^{+})\big)^{-1}\label{HFBAD_u}\\
&\qquad\times \big(
\mathbf{f} - \alpha \big(\mathbf{D}_x^{-} (\mathbf{q}^k\odot\mathbf{n}_1^k) + \mathbf{D}_y^{-} (\mathbf{q}^k\odot\mathbf{n}_2^k) \big)\big).
\nonumber
\end{align}
One readily sees that the positive definite  matrix $\mathbf{I} \otimes \mathbf{I}  + \alpha( \mathbf{D}_x^{-}\mathbf{D}_x^{+}+\mathbf{D}_y^{-}\mathbf{D}_y^{+})$ is block circulant with circulant blocks (BCCB) such that one can efficiently determine the solution using two-dimensional fast discrete Fourier transform (DFT)~\cite{Vogel2002Computational}. 
We remark that if assuming the Neumann boundary condition for ${\mathbf u}$, the resulting  matrix $\mathbf{I} \otimes \mathbf{I}  + \alpha( \mathbf{D}_x^{-}\mathbf{D}_x^{+}+\mathbf{D}_y^{-}\mathbf{D}_y^{+})$ is positive definite as well, such that one can efficiently obtain the solution of \eqref{HFBAD_u} using efficient methods such as CG and sparse Cholesky factorization.

The $\mathbf{n}$-subproblem is solved by the projection onto the sphere $\mathcal{S}$.
From \eqref{eq-auxiliary-function}, one obtain 
\[
\begin{aligned}
  &~~\nabla_{\mathbf{n}_1} E_{\alpha}(\mathbf{u}^{k+1},\mathbf{n}^k,\mathbf{q}^k)\\
  &= - 2b\mathbf{D}_x^{+}\operatorname{diag}(\mathbf{q}^k)(\mathbf{D}_x^{-}\mathbf{n}_1^k + \mathbf{D}_y^{-}\mathbf{n}_2^k)
  + \alpha \mathbf{q}^k\odot (\mathbf{q}^k\odot\mathbf{n}_1^k - \mathbf{D}_x^{+}\mathbf{u}^{k+1}),
\end{aligned}
\]
where $\operatorname{diag} (\mathbf{q}^k)$ denotes the $N^2 \times N^2$ diagonal matrix with the $i$th diagonal entry $q_i^k$.
Similarly, one has
\[
\begin{aligned}
 &\ \  \nabla_{\mathbf{n}_2} E_{\alpha}(\mathbf{u}^{k+1},\mathbf{n}^k,\mathbf{q}^k) \\
 & = - 2b\mathbf{D}_y^{+}\operatorname{diag}(\mathbf{q}^k)(\mathbf{D}_x^{-}\mathbf{n}_1^k + \mathbf{D}_y^{-}\mathbf{n}_2^k)
 + \alpha \mathbf{q}^k\odot (\mathbf{q}^k\odot\mathbf{n}_2^k - \mathbf{D}_x^{+}\mathbf{u}^{k+1}).
 \end{aligned}
\]
Then denote
\[
	\mathbf{n}^{k+\frac{1}{2}}:= \begin{bmatrix}
	                               \bfn_1^{k+\frac{1}{2}} \\
	                               \bfn_2^{k+\frac{1}{2}}
	                             \end{bmatrix}
	=\mathbf{n}^k - \tau_k \nabla_{\mathbf{n}} E_{\alpha}(\mathbf{u}^{k+1},\mathbf{n}^k,\mathbf{q}^k),
\] 
where 
\begin{align}
\mathbf{n}_1^{k+\frac{1}{2}} & = \mathbf{n}_1^k + \tau_k \big(2b \mathbf{D}_x^{+}\operatorname{diag}(\mathbf{q}^k)(\mathbf{D}_x^{-}\mathbf{n}_1^k + \mathbf{D}_y^{-}\mathbf{n}_2^k)\label{n1halfsimp}
\\
&\qquad\qquad\qquad- \alpha \mathbf{q}^k\odot (\mathbf{q}^k\odot\mathbf{n}_1^k - \mathbf{D}_x^{+}\mathbf{u}^{k+1}) \big),\nonumber
\\
\mathbf{n}_2^{k+\frac{1}{2}} & = \mathbf{n}_2^k + \tau_k
\big( 2b\mathbf{D}_y^{+}\operatorname{diag}(\mathbf{q}^k)(\mathbf{D}_x^{-}\mathbf{n}_1^k + \mathbf{D}_y^{-}\mathbf{n}_2^k)\label{n2halfsimp}
\\
&\qquad\qquad\qquad- \alpha \mathbf{q}^k\odot (\mathbf{q}^k\odot\mathbf{n}_2^k-\mathbf{D}_y^{+}\mathbf{u}^{k+1})\big).
\nonumber
\end{align}
This may not be on the unit sphere. 
We can project it onto the unit sphere by solving
  $N^2-$independent minimization problems:
\begin{equation}
\label{eq:nsubproj}
\big[\mathbf{n}^{k+1}\big]_i= \arg \min_{\mathbf{m}\in \mathcal{S}_0^1}\big\|
\mathbf{m}
- \big[\mathbf{n}^{k+\frac{1}{2}}\big]_i\big\|^2, \;
i = 1,2,\ldots,N^2,
\end{equation}
where
\[
\mathcal{S}_0^1=\{\mathbf{m}=(m_1,m_2)^{\mathrm{T}}\in \mathbb{R}^2:m_1^2 + m_2^2=1\}.
\]
Here, $\big[\mathbf{n}\big]_i = \begin{bmatrix}
n_{1,i} \\
n_{2,i}
\end{bmatrix}$, where $n_{1,i}$ and $n_{2,i}$ denote $i$th components of $\mathbf{n}_1$ and $\mathbf{n}_2$, respectively.
Each problem in \eqref{eq:nsubproj} has a closed-form solution as follows
\begin{equation}\label{HFBAD_n}
\big[\mathbf{n}^{k+1}\big]_i
= \operatorname{Proj}_{\mathcal{S}_0^1}
\big(\big[\mathbf{n}^{k+\frac{1}{2}}\big]_i\big)
:=\begin{cases}
    \frac{\big[\mathbf{n}^{k+\frac{1}{2}}\big]_i}{\big\|\big[\mathbf{n}^{k+\frac{1}{2}}\big]_i \big\|},&\text{if }\big[\mathbf{n}^{k+\frac{1}{2}}\big]_i \neq \mathbf{0},\\
    \mathbf{d},&\text{otherwise},
\end{cases}
\end{equation}
where  $\operatorname{Proj}_{\mathcal{S}_0^1}$ represents the projection operator and $\mathbf{d}$ is an arbitrary unit vector in $\mathcal{S}_0^1$.

The $\mathbf{q}$-subproblem is also separable. For all $i = 1,2,\ldots,N^2$, let
$$c_i = a + b |\mathbf{D}_{x,i}^{-}\mathbf{n}_1^{k+1}+\mathbf{D}_{y,i}^{-}\mathbf{n}_2^{k+1}|^2.$$
Then,  $q_i$-subproblem is described as
\begin{align*}
q_i^{k+1} & = \arg \min_{q_i \geq 0} \quad c_i q_i + \frac{\alpha}{2} \big( (\mathbf{D}_{x,i}^{+} \mathbf{u}^{k+1} - q_i n_{1,i}^{k+1})^2 + ( \mathbf{D}_{y,i}^{+} \mathbf{u}^{k+1} - q_i n_{2,i}^{k+1})^2 \big)\\
& = \arg \min_{q_i \geq 0} \quad c_i q_i + \frac{\alpha}{2} q_i^2  - \alpha q_i (\mathbf{D}_{x,i}^{+} \mathbf{u}^{k+1} n_{1,i}^{k+1} +  \mathbf{D}_{y,i}^{+} \mathbf{u}^{k+1}  n_{2,i}^{k+1}),
\end{align*}
where the last equation is derived based on the unit length of $\mathbf n^{k+1}$.
It is 1-dimensional quadratic programming with a nonnegative constraint, and therefore one can directly derive its closed-form solution as follows
\begin{equation}
q_i^{k+1} = \max\left(0, (\mathbf{D}_{x,i}^{+} \mathbf{u}^{k+1} n_{1,i}^{k+1} +  \mathbf{D}_{y,i}^{+} \mathbf{u}^{k+1}  n_{2,i}^{k+1}) - \frac{c_i}{\alpha}\right).
\label{HFBAD_q}
\end{equation}

We summarize the proposed HALM algorithm for solving \eqref{DisrelaxationREEM-1} in Algorithm \ref{alg:HFBAD}.

\begin{algorithm}[htbp]
	\renewcommand{\algorithmicrequire}{\textbf{Input:}}
	\renewcommand{\algorithmicensure}{\textbf{Output:}}
	\caption{HALM algorithm}
	\label{alg:HFBAD}
	\begin{algorithmic}[1]
		\REQUIRE The observation image $\mathbf{f}$, the model parameters $a,b$, the penalty parameter $\alpha$, and the step size $\{\tau_k\}$.
		\ENSURE $\mathbf{u}^k$.
		\STATE Initialization: $k=0$, 
  $\mathbf{u}^0 = \mathbf{f}$, 
  $\mathbf{n}^0 = \operatorname{Proj}_{\mathcal{S}}\big(((\mathbf{D}_x^{+}\mathbf{u}^0)^T, (\mathbf{D}_y^{+}\mathbf{u}^0)^T)^T\big)$, and $q_i^0 = \| (\mathbf{D}_{x,i}^{+}\mathbf{u}^0,~ \mathbf{D}_{y,i}^{+}\mathbf{u}^0)\|$, $i = 1,2,\ldots,N^2$.
		\WHILE{The termination condition is not satisfied}
		\STATE Compute $\mathbf{u}^{k+1}$ by \eqref{HFBAD_u} using 2D FFT.
		\STATE Compute $\mathbf{n}^{k+1}$ from \eqref{n1halfsimp}, \eqref{n2halfsimp} and \eqref{HFBAD_n}.
		\STATE Compute $\mathbf{q}^{k+1}$ via \eqref{HFBAD_q}.
        \STATE $k\leftarrow k+1$.
		\ENDWHILE
		\STATE \textbf{return} $\mathbf{u}^k$.
	\end{algorithmic}
\end{algorithm}

In the following section, we will prove the convergence of the HALM algorithm when stepsize $\tau_k$ is sufficiently small (say, satisfying \eqref{eq:stepsizecondition}). For practical use, stepsize $\tau_k$ is fixed to a constant heuristically.

\subsection{Convergence analysis of HALM}
We begin with some useful lemmas.
\begin{lemma}[Lemma 3.3 in \cite{zhangjie2020bilinear}]
\label{zhangjie'sLemma}
Let $T(\mathbf{x})= \frac{1}{2}\|\mathbf{Ax}-\mathbf{b}\|^2 + M(\mathbf{x})$, where the function $M$ is convex and the matrix $\mathbf{A}$ is $\mathbf{x}$ -independent. Suppose $\mathbf{x}^*$ is a stationary point of $T(\mathbf{x})$, i.e. $0\in\partial T(\mathbf{x}^*)$, where $\partial T(\mathbf{x}^*)$ is the subdifferential set of the function $T$ at $\mathbf{x}^*$, then we have $$T(\mathbf{x})-T(\mathbf{x}^*)\geq  \frac{1}{2} \|\mathbf{A}(\mathbf{x}-\mathbf{x}^*)\|^2.$$
\end{lemma}

\begin{lemma}[Lemma 2 in \cite{PALM}]
\label{foorwardbackwardcontrol}
Consider a composite optimization problem
\begin{align}
   \min_{\mathbf{x}}\quad \hat{F}(\mathbf{x}) = f_1(\mathbf{x}) + f_2(\mathbf{x}), \nonumber
\end{align}
where $f_2:\mathbb{R}^n\rightarrow\mathbb{R}$ is a continuously differentiable
function with gradient $\nabla f_2$ assumed $L$-Lipschitz continuous,  $f_1:\mathbb{R}^n \rightarrow ( -\infty, +\infty] $ is a proper and lower semicontinuous (maybe nonsmooth and nonconvex) function with $\inf\{f_1(\mathbf{x}):\mathbf{x}\in\mathbb{R}^n\} > -\infty$. Let $\{\mathbf{x}^k:k\in\mathbb{N}\}$ be a sequence generated by
\begin{align}
    \mathbf{x}^{k+1} &= \operatorname{Prox}_{\tau,f_1}\left(\mathbf{x}^k - \tau \nabla f_2(\mathbf{x}^k) \right) \\
    &:=\arg\min_{\mathbf{x}} f_1(\mathbf{x}) + \frac{1}{2\tau}\|\mathbf{x} - (\mathbf{x}^k - \tau \nabla f_2(\mathbf{x}^k))\|^2. \nonumber
\end{align}
Then one has 
\begin{align}
&\frac{1}{2}\left(\frac{1}{\tau}-L\right)\|\mathbf{x}^{k+1} - \mathbf{x}^k\|^2 \leq \hat{F}(\mathbf{x}^k)-\hat{F}(\mathbf{x}^{k+1}). 
\end{align}

\end{lemma}

We estimate the  Lipschitz constants of the partial derivatives of $E_{\alpha}$ for the subproblems with respect to single variable $\mathbf u$, $\mathbf q$, and $\mathbf n$,  denoted by $L_u$, $L_q$, and $L_n$,  respectively. For simplicity, we omit the superscripts and
subscripts in the following. The partial derivatives are calculated as follows
\begin{equation}
\begin{aligned}
\nabla_{\mathbf{u}} E_{\alpha}(\mathbf{u},\mathbf{n},\mathbf{q}) &= \mathbf{u} - \mathbf{f} + \alpha \big( (\mathbf{D}_x^{+})^{\mathrm{T}}(\mathbf{D}_x^{+}\mathbf{u} - \mathbf{q}\odot\mathbf{n}_1) \\
&\qquad\qquad\qquad + (\mathbf{D}_y^{+})^{\mathrm{T}}(\mathbf{D}_y^{+}\mathbf{u} - \mathbf{q}\odot\mathbf{n}_2) \big),\\
\nabla_{\mathbf{q}} E_{\alpha} (\mathbf{u},\mathbf{n},\mathbf{q})& = a{\mathbf 1} + b |\mathbf{D}_x^{-}\mathbf{n}_1 + \mathbf{D}_y^{-}\mathbf{n}_2|^2 + \alpha \big( \mathbf{n}_1\odot (\mathbf{q}\odot\mathbf{n}_1 - \mathbf{D}_x^{+}\mathbf{u}) \\
&\qquad\qquad\qquad\qquad + \mathbf{n}_2\odot (\mathbf{q}\odot\mathbf{n}_2 - \mathbf{D}_y^{+}\mathbf{u}) \big),\\
\nabla_{\mathbf{n}_1} E_{\alpha}(\mathbf{u},\mathbf{n},\mathbf{q}) & = - 2b \mathbf{D}_x^{+}\operatorname{diag}(\mathbf{q})(\mathbf{D}_x^{-}\mathbf{n}_1 + \mathbf{D}_y^{-}\mathbf{n}_2) + \alpha \mathbf{q}\odot (\mathbf{q}\odot\mathbf{n}_1 - \mathbf{D}_x^{+}\mathbf{u}),\\
\nabla_{\mathbf{n}_2} E_{\alpha}(\mathbf{u},\mathbf{n},\mathbf{q}) & = - 2b \mathbf{D}_y^{+}\operatorname{diag}(\mathbf{q})(\mathbf{D}_x^{-}\mathbf{n}_1 + \mathbf{D}_y^{-}\mathbf{n}_2) + \alpha \mathbf{q}\odot (\mathbf{q}\odot\mathbf{n}_2 - \mathbf{D}_y^{+}\mathbf{u}),
\end{aligned}
\label{eq:partialDerivatives}
\end{equation}
where $|\cdot|$ denotes the pointwise  module operation. From the first two equations in \eqref{eq:partialDerivatives}, we obtain the Lipschitz constants of $\nabla_{\mathbf{u}} E_{\alpha}$ and $\nabla_{\mathbf{q}} E_{\alpha}$ (with respect to the variables $\mathbf u$ and $\mathbf q$) as
\begin{equation}
\label{eq:Lu}
L_u = \lambda_{\max}\left(\mathbf{I} \otimes \mathbf{I} - \alpha\left(\mathbf{D}_x^{-}\mathbf{D}_x^{+}+\mathbf{D}_y^{-}\mathbf{D}_y^{+}
\right) \right)
\end{equation}
and
$$L_q(\mathbf{n}) = \alpha\|\mathbf{n}_1\odot\mathbf{n}_1 + \mathbf{n}_2\odot \mathbf{n}_2\|_{\infty},$$  where $\lambda_{\max}(\mathbf{A})$ denotes the largest eigenvalue of the matrix $\mathbf{A}$, and $\|\cdot\|_{\infty}$ denotes the $\ell^\infty$ norm in the Euclidean space.

With regard to $\nabla_{\mathbf{n}_1} E_{\alpha}$ and $\nabla_{\mathbf{n}_2} E_{\alpha}$, we consider the compound gradient using the last two equations in \eqref{eq:partialDerivatives} 
\begin{align*}
& \nabla_{\mathbf{n}}E_{\alpha}(\mathbf{u},\mathbf{n},\mathbf{q})  :=
\begin{pmatrix}
\nabla_{\mathbf{n}_1} E_{\alpha}(\mathbf{u},\mathbf{n},\mathbf{q})\\
\nabla_{\mathbf{n}_2} E_{\alpha}(\mathbf{u},\mathbf{n},\mathbf{q})
\end{pmatrix} \notag\\
& =
\underbrace{
\begin{pmatrix}
- 2b \mathbf{D}_x^{+}\operatorname{diag}(\mathbf{q})\mathbf{D}_x^{-} + \alpha\operatorname{diag}(\mathbf{q})^2 & - 2b \mathbf{D}_x^{+}\operatorname{diag}(\mathbf{q})\mathbf{D}_y^{-} \\
- 2b \mathbf{D}_y^{+}\operatorname{diag}(\mathbf{q})\mathbf{D}_x^{-} & - 2b \mathbf{D}_y^{+}\operatorname{diag}(\mathbf{q})\mathbf{D}_y^{-} + \alpha \operatorname{diag}(\mathbf{q})^2
\end{pmatrix}
}_{\mathbf{Q}(\mathbf q)}
\mathbf{n}\\
&\qquad +
\begin{pmatrix}
-\alpha\operatorname{diag}(\mathbf{q})\mathbf{D}_x^{+}\mathbf{u}\\
-\alpha\operatorname{diag}(\mathbf{q})\mathbf{D}_y^{+}\mathbf{u}
\end{pmatrix}.
\end{align*}
With $\mathbf q\geq0$ (the non-negativity condition), the matrix $\mathbf Q(\mathbf q)$ is positive semidefinite, such that one has
\[
L_n(\mathbf{q}) =  \lambda_{\max}(\mathbf{Q}(\mathbf q)) \geq 0.
\]

Next, we show that the iterative sequence generated by the proposed HALM algorithm has monotonically decreasing objective values.
\begin{lemma}
\label{Sufficient_descent}
Let the sequence  $\{(\mathbf{u}^k,\mathbf{n}^k,\mathbf{q}^k)\}_{k=1}^{+\infty}$ be generated by the HALM algorithm with variable stepsize $\{\tau_k\}_{k=1}^{+\infty}$. If 
\begin{equation}
\label{eq:stepsizecondition}
0<\tau_k \leq \frac{1}{L_n( \mathbf{q}^k)},\end{equation} the sequence $\{E_{\alpha,\mathbb{I}}(\mathbf{u}^k,\mathbf{n}^k,\mathbf{q}^k)\}_{k=1}^{+\infty}$ is nonincreasing, i.e.
\begin{align*}
& E_{\alpha,\mathbb{I}}(\mathbf{u}^k,\mathbf{n}^k,\mathbf{q}^k) - E_{\alpha,\mathbb{I}}(\mathbf{u}^{k+1},\mathbf{n}^{k+1},\mathbf{q}^{k+1}) \\
\geq & \frac{1}{2} (1+\hat{\lambda})\|\mathbf{u}^{k+1} - \mathbf{u}^k\|^2 + \frac{1}{2}\left(
\frac{1}{\tau_k} - L_n( \mathbf{q}^k) \right)  \|
\mathbf{n}^{k+1} - \mathbf{n}^k \|^2 + \frac{\alpha}{2} \|\mathbf{q}^{k+1}-\mathbf{q}^k\|^2,
\end{align*}
with $\hat{\lambda}:=\lambda_{\min}\big(-\alpha( \mathbf{D}_x^{-}\mathbf{D}_x^{+} + \mathbf{D}_y^{-}\mathbf{D}_y^{+})\big)$.
\end{lemma}
\begin{proof}
From Lemma \ref{zhangjie'sLemma} and \eqref{u_update_opt}, we obtain  
\begin{equation}
\label{energyDes_u}
E_{\alpha,\mathbb{I}}(\mathbf{u}^k,\mathbf{n}^k,\mathbf{q}^k) - E_{\alpha,\mathbb{I}}(\mathbf{u}^{k+1},\mathbf{n}^k,\mathbf{q}^k) \geq \frac{1}{2} (1+\hat{\lambda}) \|\mathbf{u}^{k+1} - \mathbf{u}^k\|^2.
\end{equation}
For $\mathbf{n}_1^{k+1}$ and $\mathbf{n}_2^{k+1}$, it follows from Lemma \ref{foorwardbackwardcontrol}  
\begin{align}
\label{energyDes_p}
& E_{\alpha,\mathbb{I}}(\mathbf{u}^{k+1},\mathbf{n}^k,\mathbf{q}^k) - E_{\alpha,\mathbb{I}}(\mathbf{u}^{k+1},\mathbf{n}^{k+1},\mathbf{q}^k) \notag \\
\geq & \frac{1}{2}\left(
\frac{1}{\tau_k} - L_n( \mathbf{q}^k) \right) \left( \|
\mathbf{n}_1^{k+1} - \mathbf{n}_1^k \|^2 + \|
\mathbf{n}_2^{k+1} - \mathbf{n}_2^k \|^2 \right).
\end{align}
With regard to $\mathbf{q}-$subproblem, for all $i=1,2,\cdots,n$, let
$$\hat{c}_i = c_i - \alpha\left(\mathbf{D}_{x,i}^{+}\mathbf{u}^{k+1}n_{1,i}^{k+1}+\mathbf{D}_{y,i}^{+}\mathbf{u}^{k+1}n_{2,i}^{k+1}\right).$$
We have
\begin{equation}
\label{qsubproblem}
\mathbf{q}^{k+1} = \arg \min_{\mathbf{q}}\quad  \frac{\alpha}{2}  \sum\nolimits_{i=1}^{N^2} \left(( q_i + \frac{\hat{c}_i}{\alpha})^2 + \mathbb{I}_{\mathbb{R}_{+}}(q_i)\right).
\end{equation}
It follows from  Lemma \ref{zhangjie'sLemma}
\begin{equation}
\label{energyDes_q}
E_{\alpha,\mathbb{I}}(\mathbf{u}^{k+1},\mathbf{n}^{k+1},\mathbf{q}^k) - E_{\alpha,\mathbb{I}}(\mathbf{u}^{k+1},\mathbf{n}^{k+1},\mathbf{q}^{k+1}) \geq \frac{\alpha}{2} \|\mathbf{q}^{k+1}-\mathbf{q}^k\|^2.
\end{equation}
Summing up relations \eqref{energyDes_u}, \eqref{energyDes_p} and  \eqref{energyDes_q}, we finally obtain the conclusion.
\end{proof}

We establish the boundedness of the iterative sequence $(\mathbf u^k, \mathbf n^k, \mathbf q^k)$ in the following lemma. 
\begin{lemma}\label{boundness_variableStep}
Let $\{(\mathbf{u}^k,\mathbf{n}^k,\mathbf{q}^k)\}_{k=1}^{+\infty}$ be generated by the  HALM algorithm with sufficiently small stepsize satisfying \eqref{eq:stepsizecondition}. Then, there exists a positive constant $C$ independent of the index $k$ such that
\begin{equation}
 \max\left\{\|\mathbf{u}^k\|, \|\mathbf n^k\|, \|\mathbf{q}^k\|\right\}\leq C.
\end{equation}
\end{lemma}

\begin{proof}
The boundedness of $\{\mathbf n^k\}$ is trivial since it is a unit vector. 
It is readily seen that
 the functional $
E_{\alpha,\mathbb{I}}(\mathbf{u},\mathbf{n},\mathbf{q})$ is coercive with respect to $(\mathbf u, \mathbf q)$. It follows from \ref{Sufficient_descent}  $E_{\alpha,\mathbb{I}}(\mathbf{u}^k,\mathbf{n}^k,\mathbf{q}^k)$ is uniformly bounded, implying that $\|\mathbf{u}^k\|_{\infty}$ and $\|\mathbf{q}^k\|_{\infty}$ are uniformly bounded. This proves the lemma.
\end{proof}

Based on the above analysis, we define 
\[\gamma =  \sup_{k\in\mathbb{N}}L_n(\mathbf{q}^k)<+\infty.\] 
In order to study the property of the limit point, we establish an upper bound for the subgradient in the following lemma.

\begin{lemma}
\label{lemma-subgradient}
There exists $\mathbf{g}^k:=(\mathbf{g}_u^k,\mathbf{g}_n^k,\mathbf{g}_q^k)$ with
\begin{align}
\mathbf{g}_u^k &:= \nabla_{\mathbf{u}}E_{\alpha}(\mathbf{u}^k,\mathbf{n}^k,\mathbf{q}^k), \nonumber\\
\mathbf{g}_n^k &\in
\nabla_{\mathbf{n}}E_{\alpha}(\mathbf{u}^k,\mathbf{n}^k,\mathbf{q}^k) + \partial\mathbb{I}_{\mathcal{S}}(\mathbf{n}^k),\nonumber\\
\mathbf{g}_q^k &\in \nabla_{\mathbf{q}}E_{\alpha}(\mathbf{u}^k,\mathbf{n}^k,\mathbf{q}^k) + \partial\mathbb{I}_{\mathbb{R}_{+}^{N^2}}(\mathbf{q}^k), \nonumber
\end{align}
such that 
\begin{align}
\|\mathbf{g}^k\| &\leq \|\mathbf{g}_u^k\| + \|\mathbf{g}_n^k\| + \|\mathbf{g}_q^k\| \label{superbound_of_gradient}\\
& \leq 2\alpha \|\mathbf{q}^k - \mathbf{q}^{k-1}\| + \left(2\gamma + \frac{1}{\tau_k}\right) \|\mathbf{n}^k - \mathbf{n}^{k-1}\|.
\nonumber
\end{align}
\end{lemma}
\begin{proof}
For $\mathbf{u}-$subproblem, from the optimality condition of the $k$th iteration we have
\begin{equation}
\mathbf{0} = \nabla_{\mathbf{u}} E_{\alpha}(\mathbf{u}^k,\mathbf{n}^{k-1},\mathbf{q}^{k-1}).
\end{equation}
Based on the estimate of the Lipschitz constant, we have
\begin{equation}
\label{lemma3_1}
\begin{aligned}
\|\mathbf{g}_u^k\| & = \|\nabla_{\mathbf{u}} E_{\alpha}(\mathbf{u}^k,\mathbf{n}^k,\mathbf{q}^k)\|  \\
& \leq \|\nabla_{\mathbf{u}} E_{\alpha}(\mathbf{u}^k,\mathbf{n}^k,\mathbf{q}^k) -
\nabla_{\mathbf{u}} E_{\alpha}(\mathbf{u}^k,\mathbf{n}^{k-1},\mathbf{q}^k)\|  \\
& \quad + \| \nabla_{\mathbf{u}} E_{\alpha}(\mathbf{u}^k,\mathbf{n}^{k-1},\mathbf{q}^k) - \nabla_{\mathbf{u}} E_{\alpha}(\mathbf{u}^k,\mathbf{n}^{k-1},\mathbf{q}^{k-1})\|  \\
& \quad +  \| \nabla_{\mathbf{u}} E_{\alpha}(\mathbf{u}^k,\mathbf{n}^{k-1},\mathbf{q}^{k-1})\|   \\
& \leq L_n(\mathbf{q}^k)\|\mathbf{n}^k - \mathbf{n}^{k-1}\| + \alpha \left\|
\mathbf{q}^k - \mathbf{q}^{k-1}
\right\|  \\
& \leq \gamma \|\mathbf{n}^k - \mathbf{n}^{k-1}\|+ \alpha \left\|
\mathbf{q}^k - \mathbf{q}^{k-1}
\right\|. 
\end{aligned}
\end{equation}

It follows from Lemma \ref{foorwardbackwardcontrol} that there exists $\mathbf{v}^k \in \partial \mathbb{I}_{\mathcal{S}}(\mathbf{n}^k)$ such that  
\begin{equation}
\|\nabla_{\mathbf{n}}E_{\alpha}(\mathbf{u}^k,\mathbf{n}^k,\mathbf{q}^{k-1}) + \mathbf{v}^k \|\leq \left( L_n( \mathbf{q}^k) + \frac{1}{\tau_k}\right)\|\mathbf{n}^k - \mathbf{n}^{k-1}\|.
\label{eq:gradnbound}
\end{equation}
Denoting $\mathbf{g}_n^k:=\nabla_{\mathbf{n}}E_{\alpha}(\mathbf{u}^k,\mathbf{n}^k,\mathbf{q}^k) + \mathbf{v}^k$ and following \eqref{eq:gradnbound},
we have
\begin{equation}
\label{lemma3_2}
\begin{aligned}
\|\mathbf{g}_n^k\|
&\leq  \| \nabla_{\mathbf{n}}E_{\alpha}(\mathbf{u}^k,\mathbf{n}^k,\mathbf{q}^k) - \nabla_{\mathbf{n}}E_{\alpha}(\mathbf{u}^k,\mathbf{n}^k,\mathbf{q}^{k-1})\|   \\
&\quad + \|\nabla_{\mathbf{n}}E_{\alpha}(\mathbf{u}^k,\mathbf{n}^k,\mathbf{q}^{k-1}) + \mathbf{v}^k \|  \\
&\leq  \alpha \|\mathbf{q}^k - \mathbf{q}^{k-1}\| + \left( L_n( \mathbf{q}^k) + \frac{1}{\tau_k}\right) \|\mathbf{n}_1^k - \mathbf{n}^{k-1}\|  \\
&\leq  \alpha \|\mathbf{q}^k - \mathbf{q}^{k-1}\| + \left(\gamma + \frac{1}{\tau_k}\right)  \|\mathbf{n}^k - \mathbf{n}^{k-1}\|.
\end{aligned}
\end{equation}

For the  $\mathbf{q}-$subproblem, we have
$$\mathbf{0} \in \nabla_{\mathbf{q}}E_{\alpha}(\mathbf{u}^k,\mathbf{n}^k,\mathbf{q}^k) + \partial\mathbb{I}_{\mathbb{R}_{+}^{N^2}}(\mathbf{q}^k).$$
Then, there exists a $\mathbf{w}^k \in \partial\mathbb{I}_{\mathbb{R}_{+}^{N^2}}(\mathbf{q}^k)$
such that 
\begin{equation}
\label{lemma3_3}
\|\mathbf{g}_q^k\|  = 0,
\end{equation}
where $\mathbf{g}_q^k:=\nabla_{\mathbf{q}}E_{\alpha}(\mathbf{u}^k,\mathbf{n}^k,\mathbf{q}^k)  + \mathbf{w}^k$.
Summing up  \eqref{lemma3_1}, \eqref{lemma3_2}, and \eqref{lemma3_3}, we arrive at the prrof of the lemma.
\end{proof}

Finally, under the condition in Lemma \ref{Sufficient_descent} and the estimate of the subgradient (bounded by the successive error) in Lemma \ref{lemma-subgradient}, 
we are ready to prove the convergence of the iterative sequence generated by the HALM algorithm. 
\begin{thm} \label{thm1}
Let $\{(\mathbf{u}^k,\mathbf{n}^k,\mathbf{q}^k):k\in\mathbb{N}\}$ be a sequence generated by the HALM algorithm with adaptive stepsize $\tau^k$ satisfying \eqref{eq:stepsizecondition}. Then,
\begin{enumerate}[(1)]
\item  sequence $\{(\mathbf{u}^k,\mathbf{n}^k,\mathbf{q}^k):k\in\mathbb{N}\}$ has a finite length, i.e.
  \[\sum_{k=1}^{+\infty} \|(\mathbf{u}^k,\mathbf{n}^k,\mathbf{q}^k)\| < + \infty;\]
\item sequence $\{(\mathbf{u}^k,\mathbf{n}^k,\mathbf{q}^k):k\in\mathbb{N}\}$ converges to a critical point of $E_{\alpha,\mathbb{I}}$.
\end{enumerate}
\end{thm}
\begin{proof}
It is trivial to show that the objective functional $E_{\alpha,\mathbb{I}}$ is a Kurdyka-\L ojasiewicz function \cite{Attouch2010Proximal, PALM}. 
The proof then follows from the proof given  in   \cite{PALM}.
\end{proof}


\section{Extension to general curvature-based  model}\label{sec-extension}

In this section, we extend the HALM algorithm  to a general curvature-based model  \cite{Duan_Min_discrete_total_curva} given in \eqref{general_model-1}.
  Some examples of function $\phi$ are given as follows
\begin{align*}
\phi(\kappa)=
\begin{cases}a+b|\kappa|, & \text{total absolute curvature (TAC)}, \\
\sqrt{a+b|\kappa|^{2}}, & \text{total rotation variation (TRV)}, \\
a+b|\kappa|^{2}, & \text{total square curvature (TSC)}.
\end{cases}
\end{align*}
Using the notations given in \cref{sec-model}, one can derive a   penalty model for the discrete form of \eqref{general_model-1} as follows
\begin{equation}
 \min_{\mathbf{u},\mathbf{n},\mathbf{q}}  E_{r,\mathbb{I}}^g(\mathbf{u},\mathbf{n},\mathbf{q}),
 \end{equation}
 where
\begin{equation}
\begin{aligned}
    E_{r,\mathbb{I}}^g(\mathbf{u},\mathbf{n},\mathbf{q}) & = \sum\nolimits_{i=1}^{N^2} \phi(\mathbf{D}_{x,i}^-\mathbf{n}_1+\mathbf{D}_{y,i}^-\mathbf{n}_2)q_i + \frac{1}{2} \left\| \mathbf{u}-\mathbf{f}\right\|^2   \\
    & \quad + \frac{\alpha}{2} \left(\|\mathbf{D}_x^+ \mathbf{u} -\mathbf{q}\odot \mathbf{n}_1\|^2 + \| \mathbf{D}_y^+ \mathbf{u} -\mathbf{q}\odot \mathbf{n}_2\|^2 \right)\\
    & \quad  + \mathbb{I}_{\mathcal{S}}(\mathbf{n}) +\mathbb{I}_{\mathbb{R}_{+}^{N^2}}(\mathbf{q})\\
    & := E_r^g(\mathbf{u},\mathbf{n},\mathbf{q}) + \mathbb{I}_{\mathcal{S}}(\mathbf{n}) + \mathbb{I}_{\mathbb{R}_{+}^{N^2}}(\mathbf{q}).
\end{aligned}
\label{DisrelaxationREEM-2}
\end{equation}

Note that the HALM algorithm's applicability extends beyond the proposed model. It can directly solve the general form in \eqref{general_model-1} given an $L-$Lipschitz smooth $\phi$. Specifically, the TSC case of \eqref{general_model-1} corresponds exactly to the EE model. Moreover, for both TRV and TSC, $\phi$ satisfies $L-$Lipschitz smoothness and $E_r^g(\mathbf{u},\mathbf{n},\mathbf{q})$ exhibits smoothness. Consequently, the associated models are amenable to direct solution via the HALM algorithm, with the generic scheme given below
\begin{equation*}
  \left\{
  \begin{aligned}
\mathbf{u}^{k+1}& =\arg\min_{\mathbf{u}}E_{r,\mathbb{I}}^g(\mathbf{u},\mathbf{n}^k,\mathbf{q}^k),\vspace{2ex}\\
\mathbf{n}^{k+1}& =\arg\min_{\mathbf{n}} \mathbb{I}_{\mathcal{S}}(\mathbf{n}) +
\frac{1}{2\tau_k} \| \mathbf{n} - \left( \mathbf{n}^k - \tau_k \nabla_{\mathbf{n}} E_r^g(\mathbf{u}^{k+1},\mathbf{n}^k,\mathbf{q}^k) \right) \|^2,\vspace{2ex}\\
\mathbf{q}^{k+1}&=\arg\min_{\mathbf{q}}E_{r,\mathbb{I}}^d(\mathbf{u}^{k+1},\mathbf{n}^{k+1},\mathbf{q}),
  \end{aligned}
  \right.
\end{equation*}
where $\tau_k$ is a positive parameter.

For example, if we use the TRV model, its specific iteration scheme is as follows.
For the $\mathbf{u}$ and $\mathbf q$ subproblems, one solves them exactly the same as in \cref{alg:HFBAD}. 
For the $\mathbf{n}$-subproblem,
\[
\mathbf{n}^{k+1}=\operatorname{Proj}_{\mathcal S}(\mathbf{n}^k - \tau_k \nabla_{\mathbf{n}}E_r^g(\mathbf{u}^{k+1},\mathbf{n}^k,\mathbf{q}^k)),
\]
where
\begin{align*}
\nabla_{\mathbf{n}_1}E_r^g &= b(\mathbf{D}_{x}^-)^{\mathrm{T}}\Bigg(\dfrac{\mathbf q^k\odot(\mathbf{D}_{x}^-\mathbf{n}_1^k +\mathbf{D}_{y}^-\mathbf{n}_2^k)}{\sqrt{a\mathbf 1 + b |\mathbf{D}_{x}^-\mathbf{n}_1+\mathbf{D}_{y}^-\mathbf{n}_2|^2}}\Bigg) \\
& \qquad\qquad\qquad + \alpha \mathbf{q}^k\odot (\mathbf{q}^k\odot\mathbf{n}_1^k - \mathbf{D}_x^+\mathbf{u}^{k+1}) ,\\
\nabla_{\mathbf{n}_2}E_r^g &= b(\mathbf{D}_{y}^-)^{\mathrm{T}}\Bigg(\dfrac{\mathbf q^k\odot(\mathbf{D}_{x}^-\mathbf{n}_1^k +\mathbf{D}_{y}^-\mathbf{n}_2^k)}{\sqrt{a\mathbf 1 + b |\mathbf{D}_{x}^-\mathbf{n}_1+\mathbf{D}_{y}^-\mathbf{n}_2|^2}}\Bigg)\\
& \qquad\qquad\qquad + \alpha \mathbf{q}^k\odot (\mathbf{q}^k\odot\mathbf{n}_2^k - \mathbf{D}_y^+\mathbf{u}^{k+1}) ,
\end{align*}
 the notation $\dfrac{~\mathbf p~}{~\mathbf q~}$ denotes the elementwise division of two vectors $\mathbf p$ and $\mathbf q$,
and
the projection operator $ \operatorname{Proj}_{\mathcal S} $ is defined by
\[
[\operatorname{Proj}_{\mathcal S} (\mathbf n)]_i:=\operatorname{Proj}_{\mathcal S_0^1} ([\mathbf n]_i)\ \ \forall \mathbf n\in\mathbb R^{2N^2},
\]
for $i = 1,2,\cdots,N^2$.

In the TRV model, $\phi$ is $L-$Lipschitz smooth with respect to variable  $\mathbf{n}$, $E_r^g(\mathbf{u},\mathbf{n},\mathbf{q})$ is convex with respect to variables  $\mathbf{u}$ and $\mathbf{q}$ respectively, and $E_{r,\mathbb{I}}^d$ is a Kurdyka-\L ojasiewicz function. For $\tau_k<\frac{1}{L}$,   the above algorithm can be shown to be convergent following the proof of Theorem \ref{thm1}.

We remark that in the nonsmooth case using TAC, a new technique should be developed to design fast convergent algorithms, e.g. introducing more auxiliary variables to deal with the nonsmooth term. We will leave it to our future work. 

\section{Numerical experiments}\label{sec-num}

In this section, we evaluate the performance of the HALM algorithm via a host of numerical experiments. We implement the HALM algorithm in MATLAB and conduct the numerical experiments on a desktop computer with an Intel i7-6700 CPU and 32GB RAM. First, we compare the algorithm with the DGT algorithm \cite{2019DengOS} and the HWC algorithm \cite{He2021Penalty} when applied to the EE model for a Gaussian denoising problem. The effects of two algorithmic  parameters are discussed as well.
Then, we extend the comparison to the speckle denoising problem.  Finally, we apply the algorithm to the TRV model for a Gaussian denoising problem.

For the algorithm, we have investigated two different ways to initialize $\mathbf{u}^{k}$: $\mathbf{u}^0 = \mathbf{0}$ and $\mathbf{u}^0 = \mathbf{f}$. The latter generates better-recovered images than the former $\mathbf{u}^0 = \mathbf{0}$.
Hence, we adopt $\mathbf{u}^0 = \mathbf{f}$ as the default  in the HALM algorithm. Unless otherwise specified, we stop the HALM algorithm when the relative error of $\mathbf{u}^{k}$ reaches the prescribed tolerance at
\begin{equation}
\text{ReErr} := \frac{\|{\mathbf{u}}^{k+1} - {\mathbf{u}}^k \|}{\|{\mathbf{u}}^k\|}< tol,
\label{eq-tolerance}
\end{equation}
where $tol = 10^{-5}$ or the iteration number reaches $N_{\text{iter}}=500$.

In our experiments, the range of the grey-scale image is limited to $[0,1]$.
We use the peak signal-to-noise ratio (PSNR) and the structural similarity index measure (SSIM) to measure the quality of the restored images. The value of PSNR  is defined as
\[
\mathrm{PSNR}(\mathbf{u},\mathbf{u}_{c}) := 10 \log_{10}\frac{ N^2}{\|\mathbf{u}-\mathbf{u}_c\|^2},
\]
where $\mathbf{u}$ and $\mathbf{u}_{c}$ are the recovered image and the ground-truth image, respectively. The value of SSIM is given by
\[
\mathrm{SSIM}(\mathbf{u},\mathbf{u}_{c}) := \frac{(2\mu_{\mathbf{u}}\mu_{\mathbf{u}_{c}} + C_{1})(2\sigma_{\mathbf{u}\mathbf{u}_{c}} + C_{2})}
{(\mu_{\mathbf{u}}^2+\mu_{\mathbf{u}_{c}}^2 + C_{1})(\sigma_{\mathbf{u}}^2+\sigma_{\mathbf{u}_{c}}^2 + C_{2})},
\]
where $\sigma_{\mathbf{u}}$, $\sigma_{\mathbf{u}_{c}}$, $\sigma_{\mathbf{u}\mathbf{u}_{c}}$, $\mu_{\mathbf{u}}$, $\mu_{\mathbf{u}_{c}}$ denote standard deviations, cross-covariance, and local means of images $\mathbf{u}$ and $\mathbf{u}_{c}$, respectively, $C_1$ and $C_2$ are two constants \cite{Wang2004SSIM}.

\subsection{Gaussian denoising by the EE model}

\begin{figure}[]
\centering
\begin{subfigure}{0.215\textwidth}
    \includegraphics[width = \linewidth]{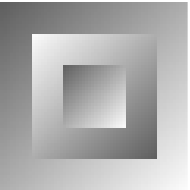}
    \caption{\#1$(60\times 60)$}
\end{subfigure}
\begin{subfigure}{0.215\textwidth}
    \includegraphics[width = \linewidth]{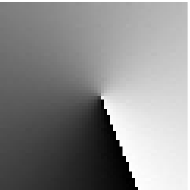}
    \caption{\#2$(60\times 60)$}
\end{subfigure}
\begin{subfigure}{0.215\textwidth}
    \includegraphics[width = \linewidth]{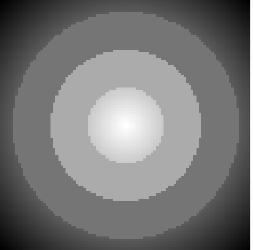}
    \caption{\#3$(100\times 100)$}
\end{subfigure}
\begin{subfigure}{0.215\textwidth}
    \includegraphics[width = \linewidth]{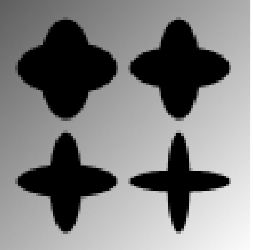}
    \caption{\#4$(100\times 100)$}
\end{subfigure}\\ \vspace{-1em}
\begin{subfigure}{0.215\textwidth}
    \includegraphics[width = \linewidth]{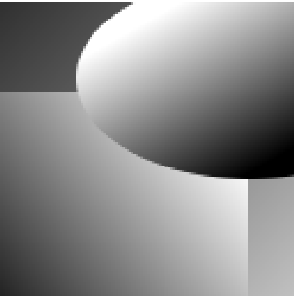}
    \caption{\#5$(128\times 128)$}
\end{subfigure}
\begin{subfigure}{0.215\textwidth}
    \includegraphics[width = \linewidth]{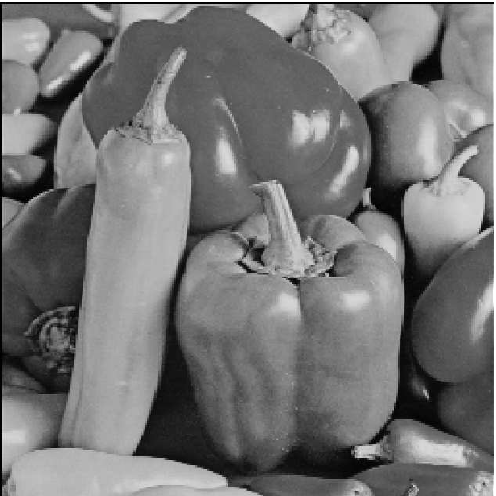}
    \caption{\#6$(256\times 256)$}
\end{subfigure}
\begin{subfigure}{0.215\textwidth}
    \includegraphics[width = \linewidth]{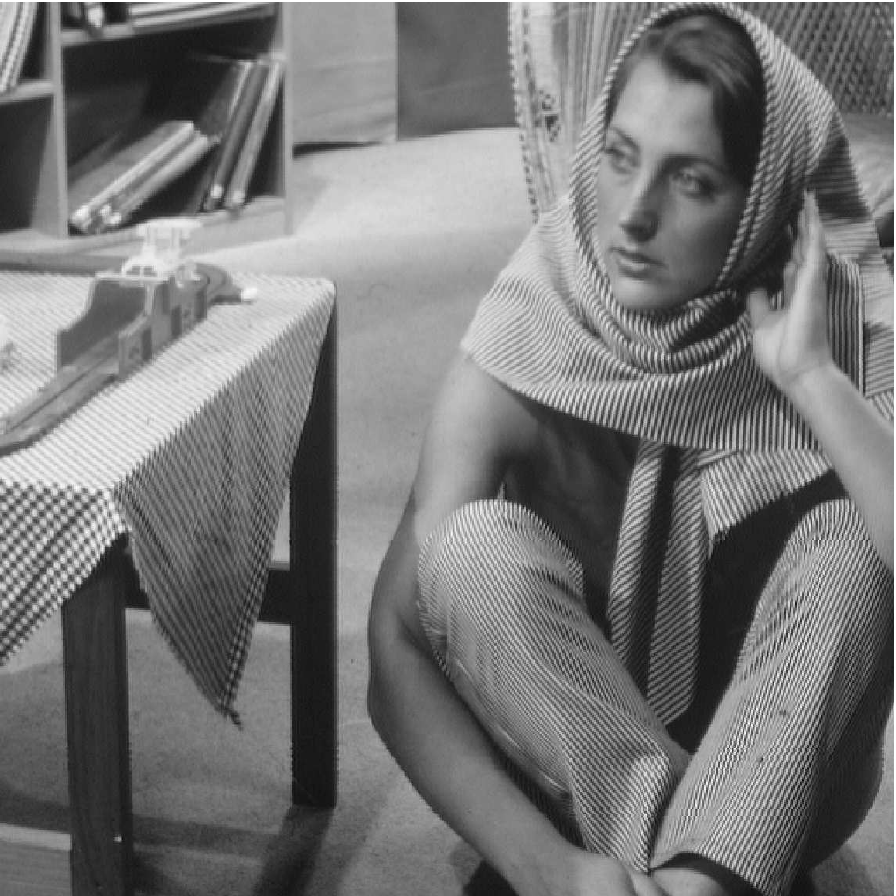}
    \caption{\#7$(512\times 512)$}
\end{subfigure}
\begin{subfigure}{0.215\textwidth}
    \includegraphics[width = \linewidth]{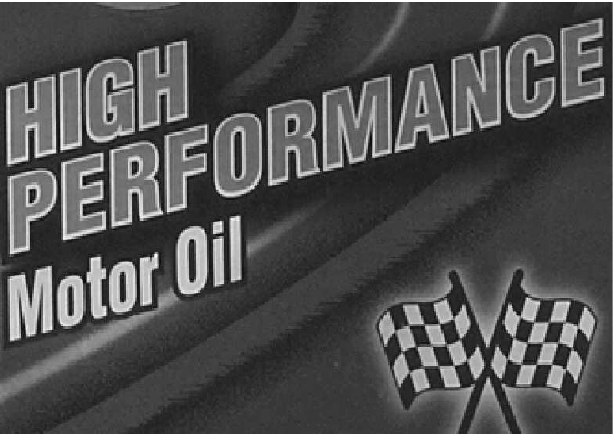}
    \caption{\#8$(332\times 216)$}
\end{subfigure} \vspace{-1em}
  \caption{Ground-truth images for Gaussian denoising.}
  \label{groundtruth_images}
\end{figure}

\begin{table}[]
	\caption{Comparison of the PSNR and SSIM values among the three algorithms.}
	\label{tab-comparison-PSNR-and-SSIM}
    \begin{center}
         \begin{tabular}{|c|c|c|c|c|c|c|}
         \hline
         \multirow{2}[0]{*}{Image} & \multicolumn{3}{c|}{PSNR} & \multicolumn{3}{c|}{SSIM} \\
     \cline{2-7}          & DGT   & HWC   & HALM  & DGT   & HWC   & HALM \\
         \hline
         \#1     & 30.86  & 33.04  & 32.93  & 0.9303  & 0.9348  & 0.9317  \\
         \hline
         \#2     & 33.06  & 35.13  & 34.63  & 0.9511  & 0.9587  & 0.9664  \\
         \hline
         \#3     & 34.88  & 36.33  & 37.33  & 0.9518  & 0.9582  & 0.9629  \\
         \hline
         \#4     & 31.54  & 32.38  & 34.48  & 0.8565  & 0.8613  & 0.8795  \\
         \hline
         \#5     & 35.57  & 38.08  & 38.09  & 0.9703  & 0.9695  & 0.9697  \\
         \hline
         \#6     & 33.16  &  $/$  & 33.10  & 0.9071  & $/$  & 0.9048  \\
         \hline
         \#7     & 31.04  &   $/$   & 31.31  & 0.8509  &    $/$   & 0.8690  \\
         \hline
         \#8     & 32.47  &   $/$    & 32.39  & 0.9158  &    $/$   & 0.9019  \\
         \hline
         \end{tabular}%
    \end{center}
\end{table}%

\newcolumntype{R}{>{$}r<{$}}
\begin{table}[]
    \caption{Computational cost comparison among the three algorithms.}
    \label{tab-comparison-computational-costs}
    \begin{center}
         \begin{tabular}{|c|c|S[table-format=3.2]|S[table-format=3.2]|S[table-format=3.2]|}
         \hline
         Image & Method & \multicolumn{1}{c|}{Iterations} & \multicolumn{1}{c|}{Time (sec)} & \multicolumn{1}{c|}{Average time per iteration} \\
         \hline
         \multirow{3}[0]{*}{\#1 ($60\times60$)} & DGT   & 171   & 0.4686 & 0.0027  \\
     \cline{2-5}          & HWC   & 132   & 68.3448 & 0.5178  \\
     \cline{2-5}          & HALM  & 119   & 0.2099 & 0.0018  \\
         \hline
         \multirow{3}[0]{*}{\#2 ($60\times60$)} & DGT   & 262   & 0.6378 & 0.0024  \\
     \cline{2-5}          & HWC   & 113   & 59.8544 & 0.5297  \\
     \cline{2-5}          & HALM  & 152    & 0.2656 & 0.0017  \\
         \hline
         \multirow{3}[0]{*}{\#3 ($100\times100$)} & DGT   & 319   & 1.8379 & 0.0058  \\
     \cline{2-5}          & HWC   & 108   & 213.2229 & 1.9743  \\
     \cline{2-5}          & HALM  & 220   & 0.5803 & 0.0026  \\
         \hline
         \multirow{3}[0]{*}{\#4 ($100\times100$)} & DGT   & 447   & 2.5982 & 0.0058  \\
     \cline{2-5}          & HWC   & 112   & 211.2046 & 1.8858  \\
     \cline{2-5}          & HALM  & 168   & 0.4346 & 0.0026  \\
         \hline
         \multirow{3}[0]{*}{\#5 ($128\times128$)} & DGT   & 287   & 2.3686 & 0.0083  \\
     \cline{2-5}          & HWC   & 90    & 386.0862 & 4.2898  \\
     \cline{2-5}          & HALM  & 133   & 0.4714 & 0.0035  \\
         \hline
         \multirow{3}[0]{*}{\#6 ($256\times256$)} & DGT   & 357   & 8.9341 & 0.0250  \\
     \cline{2-5}          & HWC   & $/$   & $/$ & $/$  \\
     \cline{2-5}          & HALM  & 184   & 2.6032 & 0.0142  \\
         \hline
         \multirow{3}[0]{*}{\#7 ($512\times512$)} & DGT   & 256   & 38.7484 & 0.1514  \\
     \cline{2-5}          & HWC   &    $/$   &    $/$   &   $/$\\
     \cline{2-5}          & HALM  & 125   & 8.4747 & 0.0678  \\
         \hline
         \multirow{3}[0]{*}{\#8 ($216\times332$)} & DGT   & 384   & 12.1821 & 0.0317  \\
     \cline{2-5}          & HWC   &   $/$    &   $/$    &  $/$\\
     \cline{2-5}          & HALM  & 238   & 3.7672 & 0.0158  \\
         \hline
\multirow{3}[0]{*}{Average} & DGT   & 310   & 8.4720 & 0.0273  \\
     \cline{2-5}          & HWC   &   $/$    &   $/$    &  $/$\\
     \cline{2-5}          & HALM  & 167   & 2.1009 & 0.0126  \\
         \hline
         \end{tabular}%
    \end{center}
\end{table}%

\begin{figure}[]
\centering
\begin{subfigure}{0.24\textwidth}
    \includegraphics[width = \linewidth]{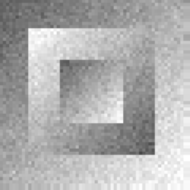}
    \caption{Noisy}
\end{subfigure}
\begin{subfigure}{0.24\textwidth}
    \includegraphics[width = \linewidth]{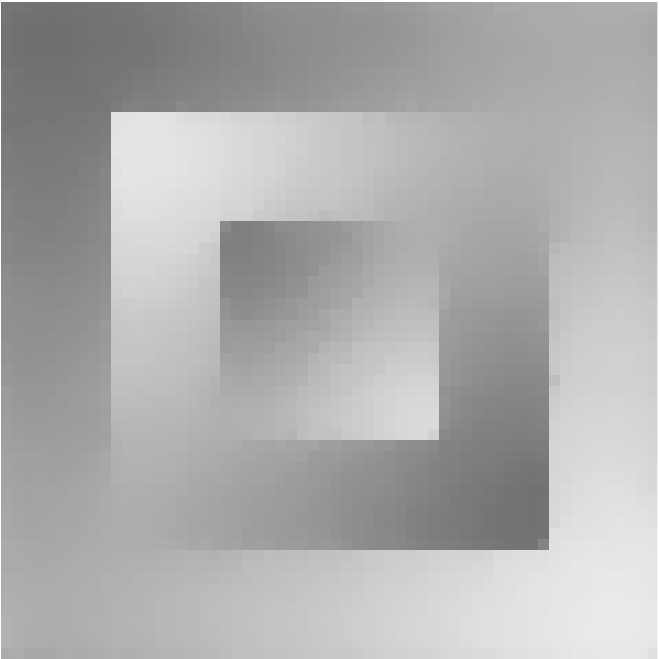}
    \caption{DGT}
\end{subfigure}
\begin{subfigure}{0.24\textwidth}
    \includegraphics[width = \linewidth]{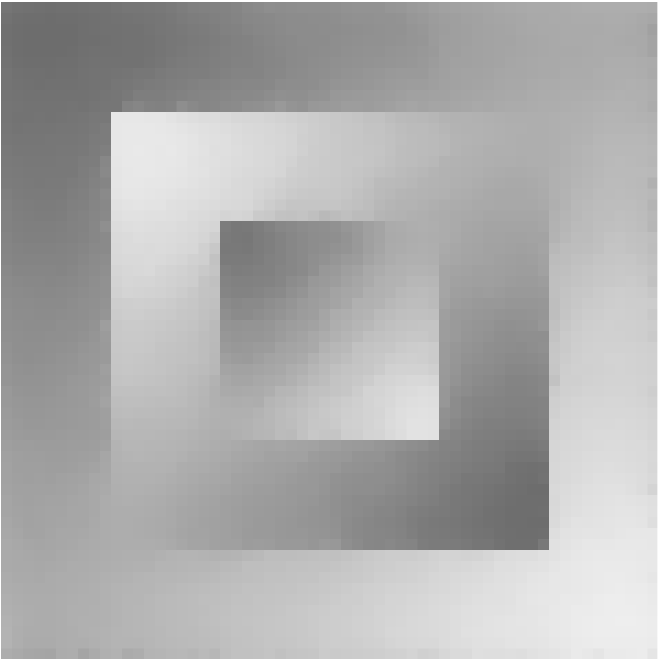}
    \caption{HWC}
\end{subfigure}
\begin{subfigure}{0.24\textwidth}
    \includegraphics[width = \linewidth]{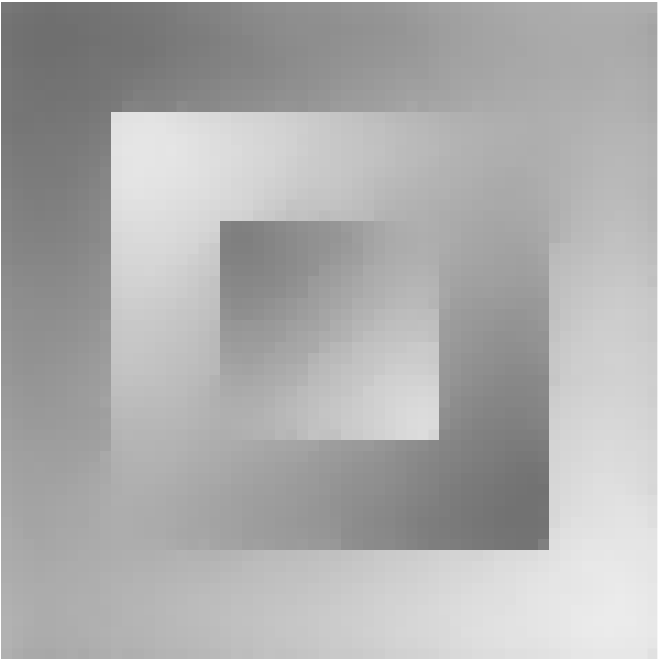}
    \caption{HALM}
\end{subfigure} \vspace{-1em}
\caption{Gaussian denoising results by solving the EE model. (a) The noisy ``\#1"  image with the additive Gaussian noise following $\mathcal{N}(0,0.0015)$. The denoised images: (b) the DGT algorithm, (c) the HWC  algorithm, (d) the proposed  HALM algorithm.}
\label{fig-HALM-denoising}
\end{figure}

\begin{figure}[]
    \centering
    \includegraphics[width = 0.6\textwidth]{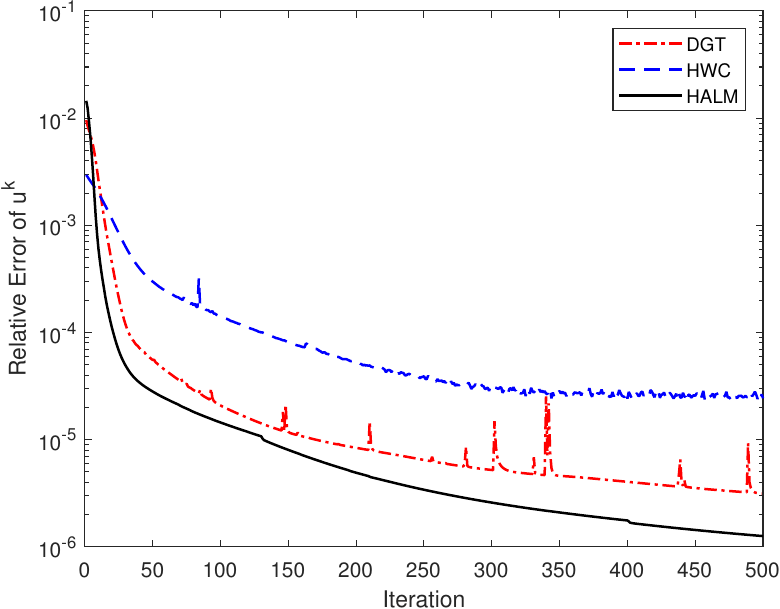}
    \caption{Relative errors of the  DGT, HWC, and HALM algorithms. This example is carried out on the \#1 image.}
    \label{fig-RelErrComparison}
\end{figure}


In this numerical experiment, we compare the HALM algorithm with two state-of-art algorithms,  the DGT and HWC algorithms, to solve the EE model.  The eight ground-truth images are shown in \cref{groundtruth_images}. Images \#1 to \#5 are synthetic smoothed images, and images \#6 to \#8 are natural images. We add random noises following the normal distribution $\mathcal{N}(0,0.0015)$ to these images and then employ the HALM, DGT, and HWC algorithms to denoise the degraded images. The source MATLAB codes of the DGT and HWC algorithms are kindly provided by the authors of \cite{2019DengOS} and \cite{He2021Penalty}, respectively.
We set the same maximum number of iterations for the HALM, DGT, and HWC algorithms at $N_{\text{iter}}=500$.
The tolerance of the relative error of $\mathbf{u}^{k}$ given by \eqref{eq-tolerance} for the HALM and DGT algorithms is $tol = 10^{-5}$ and that for HWC is $tol = 10^{-4}$. Note that any algorithm that runs for more than 1000 seconds in our experiment will be forced to terminate.

For a fair comparison, we adopt the suggestions for parameter setting of the DGT algorithm in \cite{2019DengOS} and the HWC algorithm in \cite{He2021Penalty}. It is worth noting that both the DGT and HALM algorithms have four parameters, while the HWC algorithm has nine.
 We use the strategy of \cite{2019DengOS} for parameter $\gamma$ of the DGT algorithm  and set parameters at $c=0.01$, $d=100$, $\epsilon = 0.01$, $\mu = 0.1$ and $\theta = 0.9$ for the HWC algorithm. We choose $\tau_{k}=\tau = 0.1$ in the HALM algorithm. For other parameters in the HALM, DGT, and HWC algorithms, we manually tune them to achieve better PSNR values.

\Cref{tab-comparison-PSNR-and-SSIM} shows the comparison of the PSNR and SSIM values among the three algorithms.
The results obtained from the HWC algorithm for images \#6, \#7, and \#8 are missing because the algorithm has been running for over 1000 seconds and hasn't reached the stop criterion.
We observe that the HALM algorithm generates comparable high-quality restored images like the DGT and HWC algorithms. The denoised images generated by the HALM, DGT, and HWC algorithms for the \#1 image are shown in \Cref{fig-HALM-denoising}.

We record the computational costs, including the total number of iterations and the CPU times, in \cref{tab-comparison-computational-costs}. 
The HALM algorithm is the fastest, followed by the DGT and HWC algorithms, respectively. Especially 
due to the lower computational cost in each iteration and the fewer iterations needed in the  HALM algorithm,  the average running time in this method is about $1/4$ of that in the DGT algorithm.
The HWC algorithm is the slowest due to the cost of updating the coefficient matrices in each iteration at each pixel.
The computational time of the HWC algorithm increases rapidly with the number of image pixels.  

\Cref{fig-RelErrComparison} presents the evolution of relative errors of $\mathbf{u}^{k}$  using the DGT, HWC, and HALM algorithms for restoring the \#1 image. The relative error of the HALM algorithm decreases as the number of iterations increases, numerically validating the convergence of the algorithm,
while the relative error of the DGT algorithm oscillates when its number of iterations exceeds 100.
We can also see that the relative error of the HWC algorithm falls off rapidly in the initial 150 iterations, but its numerical convergence rate is lower than that of HALM.

\begin{figure}[]
    \centering
\begin{tabular}{@{}c@{\hskip 1ex}c@{}}    
    \begin{subfigure}{0.485\textwidth}
        \includegraphics[width = \linewidth]{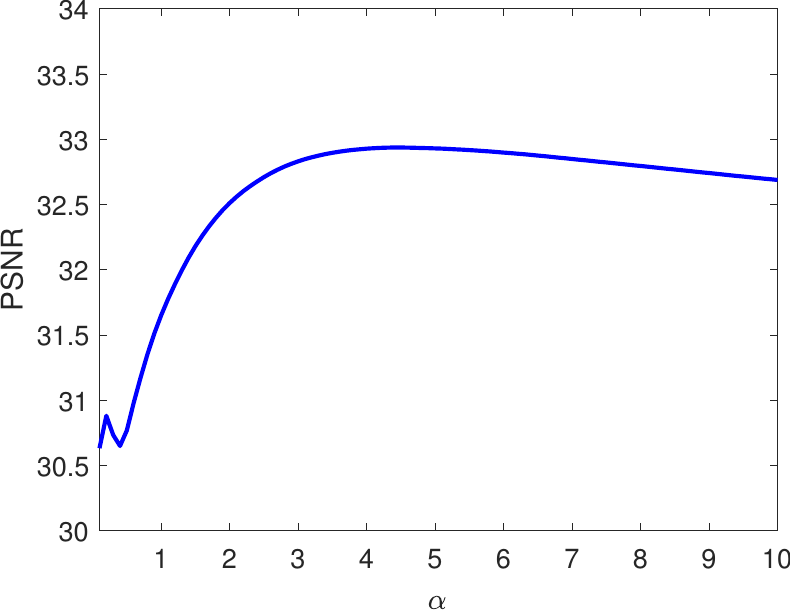}
    \end{subfigure} &
    \begin{subfigure}{0.48\textwidth}
        \includegraphics[width = \linewidth]{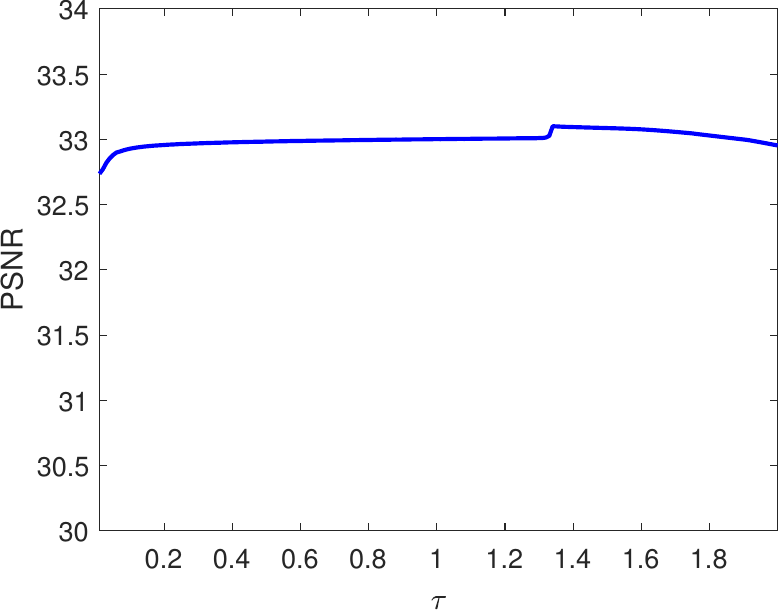}
    \end{subfigure}
\end{tabular}    
    \caption{Performances of the proposed HALM under different $\alpha$ and $\tau$. This example is carried out on the \#1 image.}
    \label{fig-Parameter-test}
\end{figure}

At the end of this part, we conduct the experiment to discuss the algorithmic parameters $\alpha$ and $\tau$ in the HALM algorithm.
\Cref{fig-Parameter-test}(a) shows the PSNR values of the recovered \#1 image using the HALM algorithm with $\alpha \in [0.1,10]$. The penalty parameter $\alpha$ in \eqref{DisrelaxationREEM-1} should be big enough inferred from \cref{fig-Parameter-test}(a),  which is consistent with the discussion at the end of  section \ref{sec-model}.
We suggest use parameter $\alpha$ in $[1,10]$.
We also test parameter $\tau \in [0.01,2]$ and
show the PSNR values of the recovered \#1 image in \cref{fig-Parameter-test}(b). Basically, the parameter $\tau$ should be chosen small enough following the condition in \cref{Sufficient_descent}. 
Here, we empirically choose $\tau \in [0.1, 1.5]$ with $\tau = 0.1$ as the default.

\subsection{Extensive tests} 

This subsection presents more experiments to demonstrate various aspects of the HALM algorithm. These include examining the effects of different boundary conditions and evaluating the performance of denoising binary, and real color images.

\begin{figure}[]
\centering
\begin{tabular}{@{}c@{\hskip 0.5ex}c@{\hskip 0.5ex}c@{\hskip 0.5ex}c@{}}
\begin{subfigure}{0.24\textwidth}
    \includegraphics[width = \linewidth]{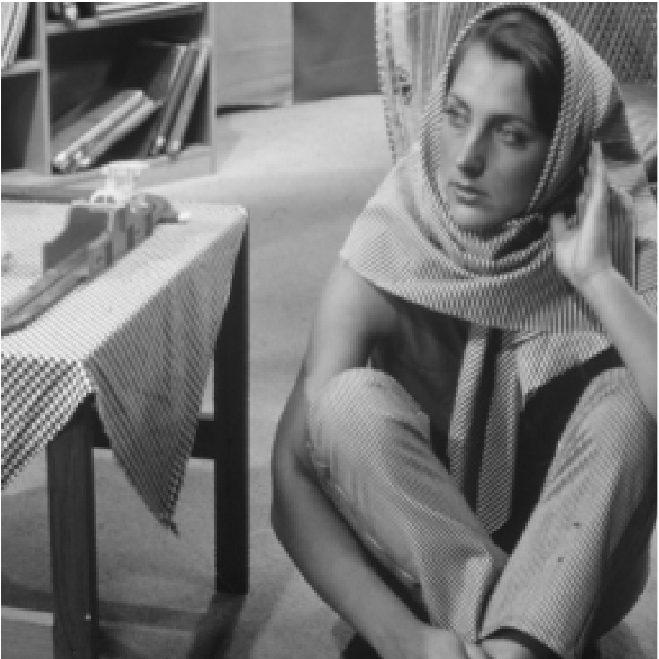}
    \caption{Ground-truth}
\end{subfigure} &
\begin{subfigure}{0.24\textwidth}
    \includegraphics[width = \linewidth]{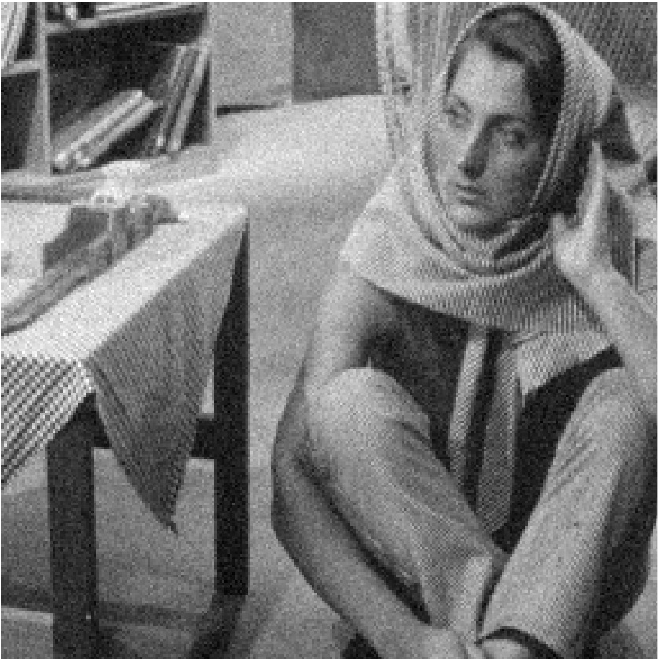}
    \caption{Noisy}
\end{subfigure} &
\begin{subfigure}{0.24\textwidth}
    \includegraphics[width = \linewidth]{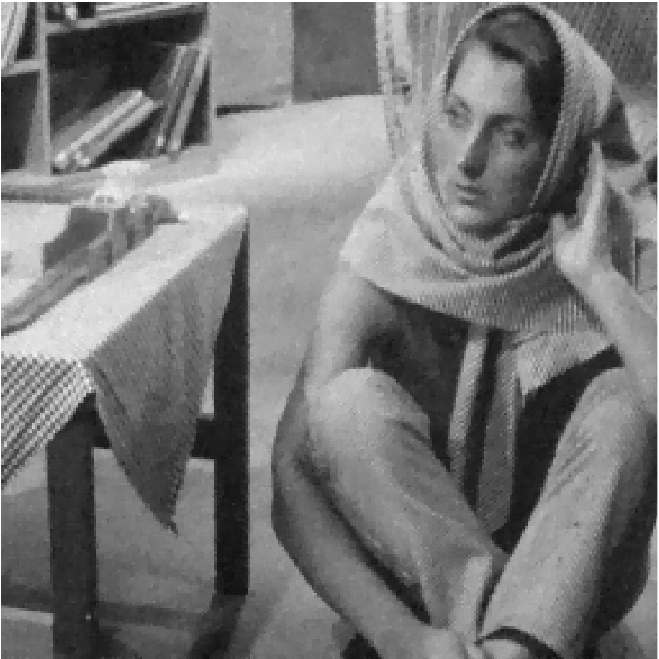}
    \caption{FFT. Time: 3.82s}
\end{subfigure} &
\begin{subfigure}{0.24\textwidth}
    \includegraphics[width = \linewidth]{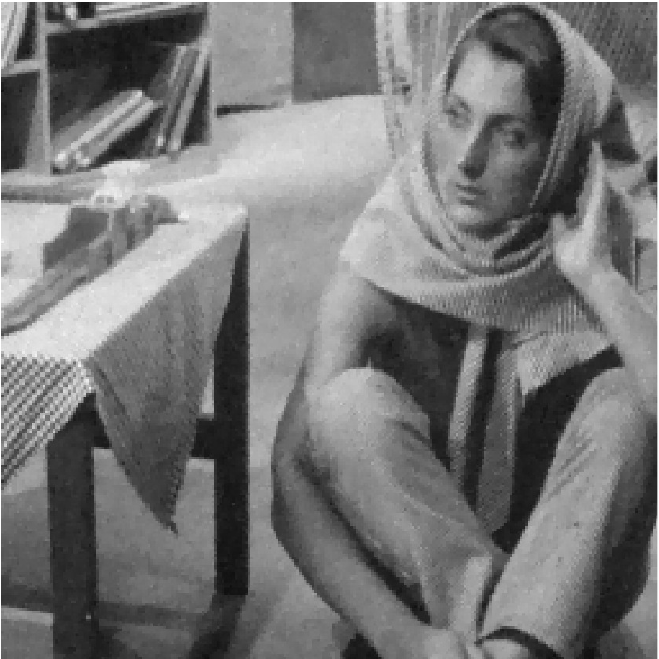}
    \caption{CG. Time: 5.64s}
\end{subfigure}
\end{tabular} \vspace{-1em}
\caption{Gaussian denoising results by solving the EE model with different boundary conditions. (a) Ground-truth ``barbara"  image $(256\times 256)$. (b) Noisy image with additive Gaussian noise (mean 0, variance 0.0015). (c) Image denoised by HALM algorithm under the periodic boundary condition with PSNR/SSIM values of 31.76/0.8866. (d) Image denoised by HALM algorithm under the Neumann boundary condition with PSNR/SSIM values of 31.78/0.8874.}
\label{fig-HALM-FFTvsCG}
\end{figure}

In the first experiment, we discuss the effects of different boundary conditions when using the HALM algorithm to solve the EE model. We compare the performance of the HALM algorithm under periodic and Neumann boundary conditions. For the Neumann case, we solve the $u$-subproblem  using the CG method. The denoising results for the ``barbara"  image $(256\times 256)$ are shown in \Cref{fig-HALM-FFTvsCG}. The denoised image under the periodic boundary condition achieves a PSNR of 31.76 and an SSIM of 0.8866, while the Neumann boundary condition result has a PSNR of 31.78 and an SSIM of 0.8874. The HALM algorithm produces comparable quantitative results with both boundary conditions. However, the periodic case is faster, with a CPU time of 3.82s, versus 5.64s for the Neumann case. This is because the periodic boundary condition allows using FFT to solve the $u$-subproblem.

\begin{figure}[]
\centering
\begin{tabular}{@{}c@{\hskip 0.5ex}c@{\hskip 0.5ex}c@{\hskip 0.5ex}c@{}}
\begin{subfigure}{0.24\textwidth}
    \includegraphics[width = \linewidth]{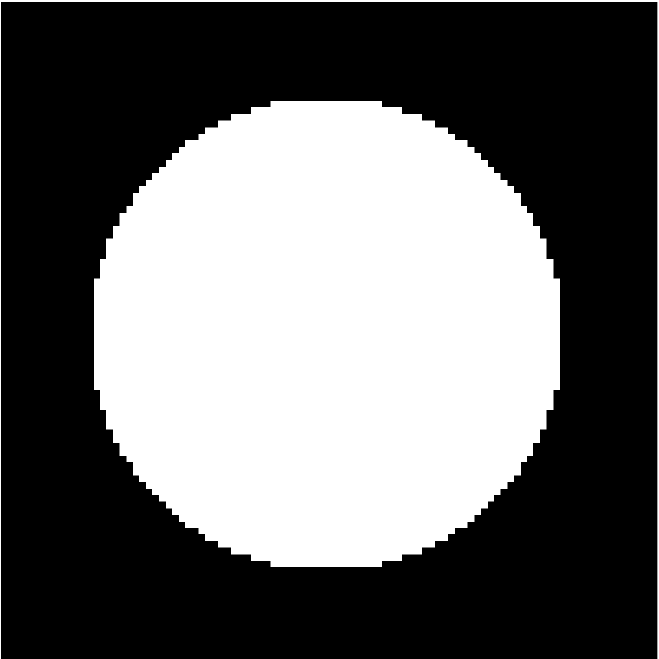}
    \caption{Ground-truth}
\end{subfigure} &
\begin{subfigure}{0.24\textwidth}
    \includegraphics[width = \linewidth]{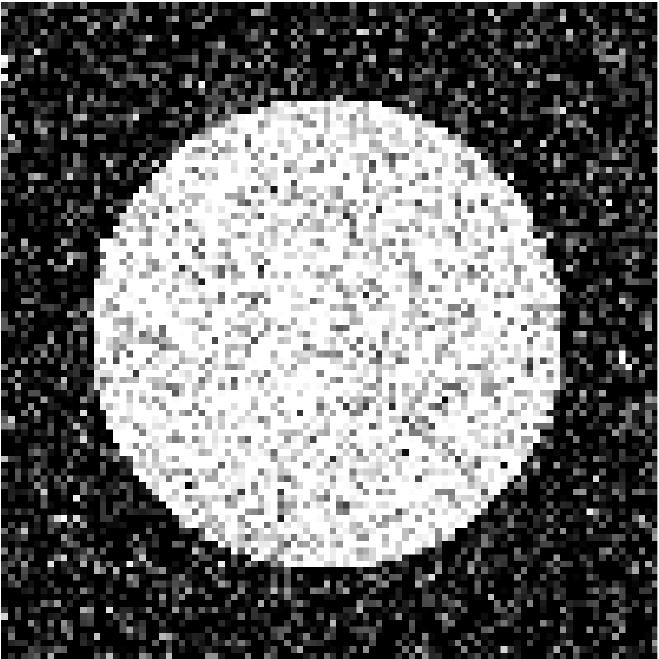}
    \caption{Noisy}
\end{subfigure} &
\begin{subfigure}{0.24\textwidth}
    \includegraphics[width = \linewidth]{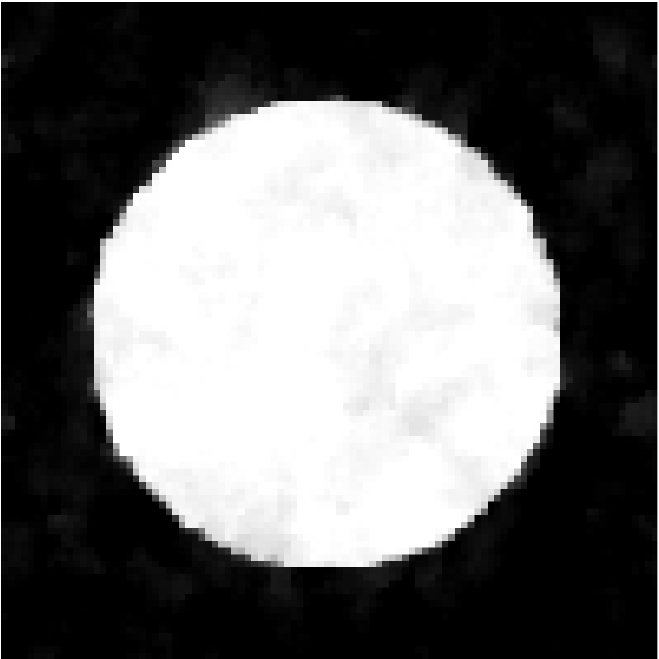}
    \caption{DGT}
\end{subfigure} &
\begin{subfigure}{0.24\textwidth}
    \includegraphics[width = \linewidth]{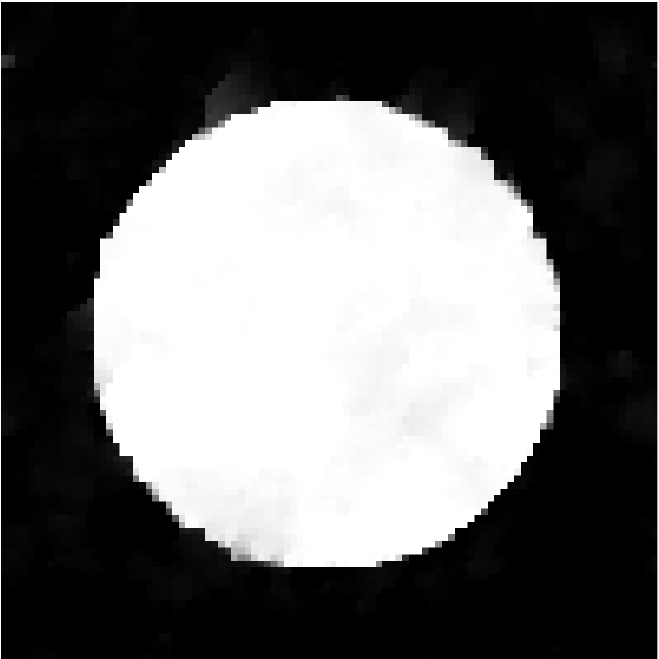}
    \caption{HALM}
\end{subfigure}
\end{tabular} \vspace{-1em}
\caption{Gaussian denoising results of binary images by solving EE model. (a) Ground-truth ``circle" image $(100\times 100)$. (b) Noisy image with additive Gaussian noise (mean 0, variance 0.1), with PSNR/SSIM values of 10.07/0.0953. (c) Image denoised by DGT algorithm with PSNR/SSIM values of 25.27/0.7358. (d) Image denoised by HALM algorithm with PSNR/SSIM values 26.07/0.8194.}
\label{fig-Binary-denoising}
\end{figure}

In the second experiment, we apply the HALM algorithm to the EE model to denoise binary images and compare them with the DGT algorithm. We test on a $(100\times 100)$ ``circle" image corrupted by the additive Gaussian noise following $\mathcal{N}(0,0.1)$, as shown in \Cref{fig-Binary-denoising} (a) and (b). The denoising results by the DGT and HALM algorithms are displayed in \Cref{fig-Binary-denoising} (c) and (d). The denoised image by the DGT algorithm has a PSNR of 25.27 and an SSIM of 0.7358, while the HALM algorithm result achieves a PSNR of 26.07 and an SSIM of 0.8194. As the HALM algorithm is a level-set approach, it can effectively preserve the circle shape while removing noise. This demonstrates an advantage over DGT for binary image denoising.

\begin{figure}[]
\centering
\begin{tabular}{c@{\hskip 0.5ex}c@{\hskip 0.5ex}c}
  \begin{subfigure}{0.24\textwidth}
    \includegraphics[width = \linewidth]{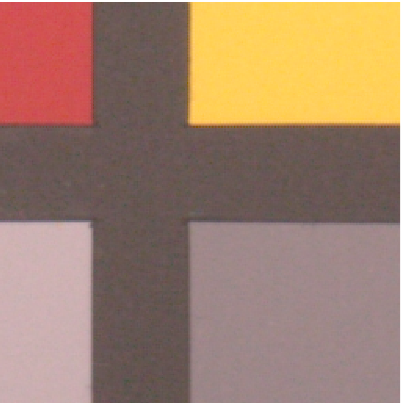}
    \caption{Ground-truth}
\end{subfigure} & 
\begin{subfigure}{0.24\textwidth}
    \includegraphics[width = \linewidth]{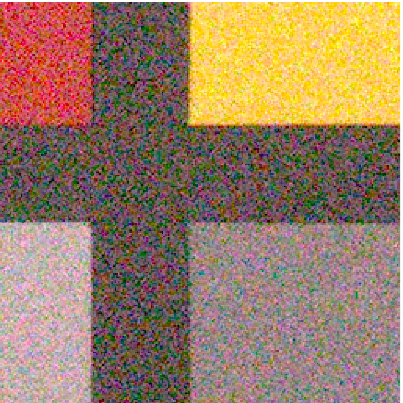}
    \caption{Noisy}
\end{subfigure} & 
\begin{subfigure}{0.24\textwidth}
    \includegraphics[width = \linewidth]{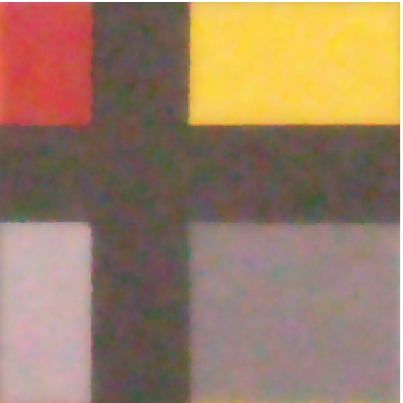}
    \caption{HALM}
\end{subfigure} \\ [-1em]
\begin{subfigure}{0.24\textwidth}
    \includegraphics[width = \linewidth]{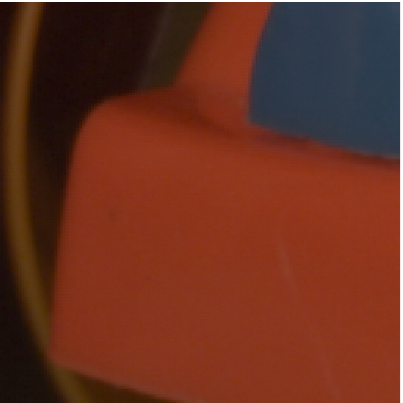}
    \caption{Ground-truth}
\end{subfigure} & 
\begin{subfigure}{0.24\textwidth}
    \includegraphics[width = \linewidth]{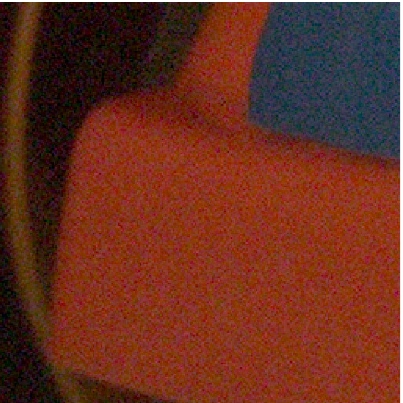}
    \caption{Noisy}
\end{subfigure} & 
\begin{subfigure}{0.24\textwidth}
    \includegraphics[width = \linewidth]{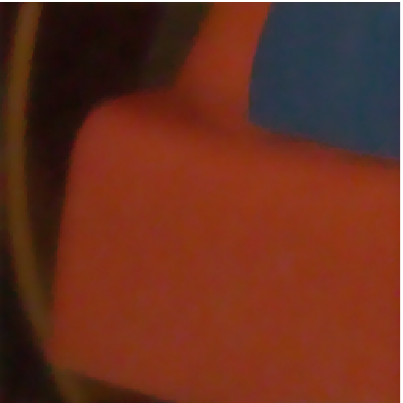}
    \caption{HALM}
\end{subfigure}
\end{tabular}
\vspace{-1em}
\caption{Denoising results of real noisy color images by solving EE model with the HALM algorithm. (a) and (d) show the ground-truth images $(256\times 256)$. (b) and (e) show the noisy images, with PSNR/SSIM values of 18.39/0.4102 and 26.27/0.8755, respectively. (c) and (f) display the denoised images with improved PSNR/SSIM values of 28.69/0.8879 and 36.58/0.9875, respectively.}
\label{fig-Real-Color-denoising}
\end{figure}

\begin{table}[]
	\caption{Denoising results in average PSNR and SSIM values on the SIDD}
	\label{tab-SIDD-PSNR-and-SSIM}
    \begin{center}
         \begin{tabular}{|c|c|c|c|c|}
         \hline
         \multirow{2}[0]{*}{Method} & \multicolumn{2}{c|}{SIDD Validation} & \multicolumn{2}{c|}{SIDD Benchmark} \\
     \cline{2-5}         
     & PSNR   & SSIM   & PSNR   & SSIM \\ \hline
     DGT & 33.04 & 0.8895 & 33.00 & 0.8860\\ \hline
     HALM & 33.22 & 0.8894 & 33.17 & 0.8850\\ \hline
         \end{tabular}%
    \end{center}
\end{table}%

In the third experiment, we employ the HALM algorithm to the EE model to restore real noisy images in standard RGB color spaces
from the Smartphone Image Denoising Dataset (SIDD) \cite{SIDD_2018_CVPR}. 
We first denoise each color channel (R, G, B) separately and then recombine the channels to obtain the final denoised color image. 
\Cref{tab-SIDD-PSNR-and-SSIM} reports the denoising results in terms of average PSNR and SSIM values over the 1280 noisy image blocks from the SIDD validation data and the 1280 blocks from the SIDD benchmark data. \Cref{fig-Real-Color-denoising} (b) and \Cref{fig-Real-Color-denoising} (e) display two real noisy images from the SIDD validation data, while \Cref{fig-Real-Color-denoising} (a) and \Cref{fig-Real-Color-denoising}  (d) show their corresponding ground-truth images. \Cref{fig-Real-Color-denoising} (c) presents one denoised image with a PSNR value of 26.27 and SSIM value of 0.8755. This denoised image exhibits much-improved quality compared to the original noisy image in \Cref{fig-Real-Color-denoising} (b), which has a PSNR of 19.39 and SSIM of 0.4012. The other denoised image in \Cref{fig-Real-Color-denoising} (f) with a PSNR of 36.58 and SSIM of 0.9875 also demonstrates significantly enhanced quality over the original noisy image in \Cref{fig-Real-Color-denoising} (e), which has a PSNR of 28.69 and SSIM of 0.8879. As shown in \Cref{fig-Real-Color-denoising}, the HALM algorithm effectively removes noise while preserving features. However, some artificial mosaics can be observed around the edges in the restored image \Cref{fig-Real-Color-denoising} (c), likely due to the independent channel-by-channel denoising process. To address this issue, a promising approach would be denoising multiple channels concurrently. We leave the exploration of joint multi-channel denoising with EE energy as future work.

\subsection{Speckle denoising}

\begin{figure}[]
    \centering
\begin{tabular}{@{}c@{\hskip 0.5ex}c@{\hskip 0.5ex}c@{\hskip 0.5ex}c@{}}
      \begin{subfigure}{ 0.24\textwidth}
      \begin{overpic}[width = \linewidth]{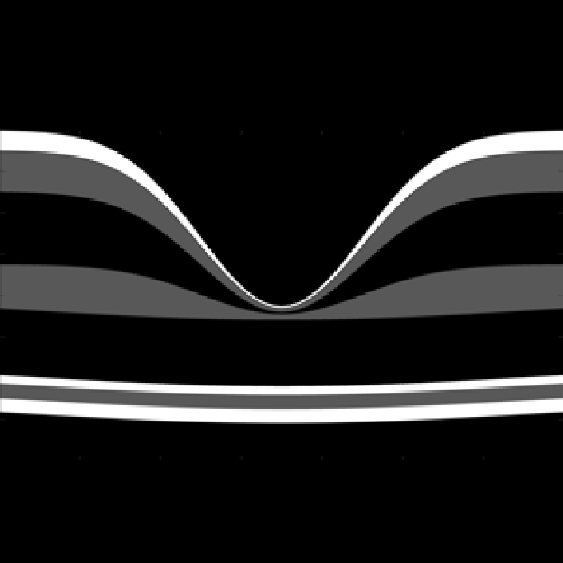}
      \put(40,40){\tikz\draw[red,thick,dashed] (0,0) rectangle (.60,.45);}
      \end{overpic}
      \caption{Ground-truth}
      \label{fig-synoct_images-a}
      \end{subfigure} &
      \begin{subfigure}{0.24\textwidth}
      \begin{overpic}[width = \linewidth]{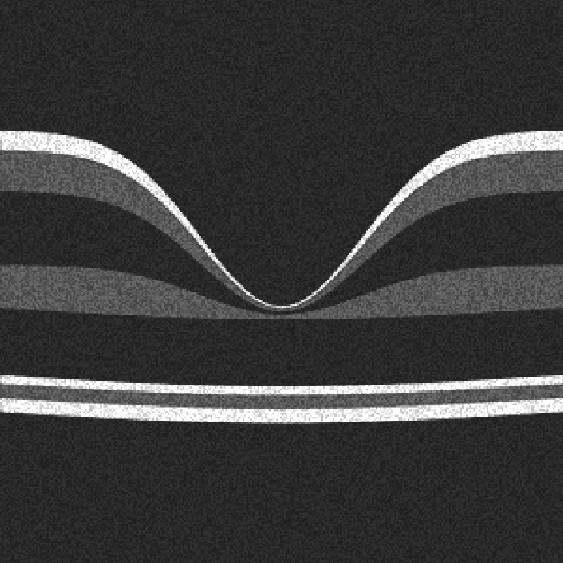}
          \put(40,40){\tikz\draw[red,thick,dashed] (0,0) rectangle (.60,.45);}
      \end{overpic}
      \caption{Noisy}
      \label{fig-synoct_images-b}
      \end{subfigure} &
      \begin{subfigure}{0.24\textwidth}
      \begin{overpic}[width = \linewidth]{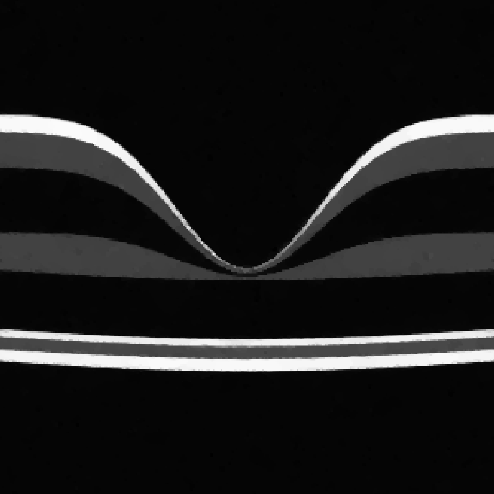}
          \put(40,40){\tikz\draw[red,thick,dashed] (0,0) rectangle (.60,.45);}
      \end{overpic}
      \caption{HWC}
      \label{fig-synoct_images-c}
      \end{subfigure} &
      \begin{subfigure}{0.24\textwidth}
      \begin{overpic}[width = \linewidth]{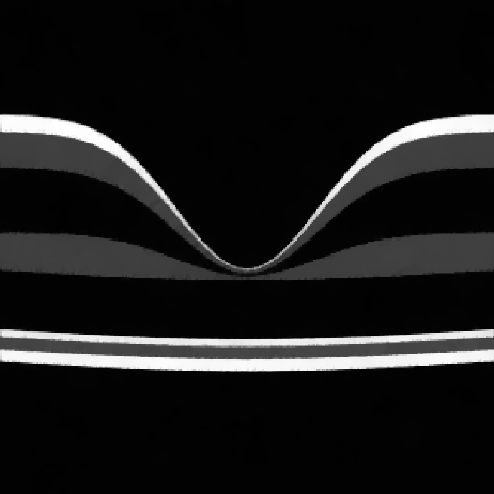}
          \put(40,40){\tikz\draw[red,thick,dashed] (0,0) rectangle (.60,.45);}
      \end{overpic}
      \caption{HALM} 
      \label{fig-synoct_images-d}
      \end{subfigure} \\[-1em]
      \begin{subfigure}{0.24\textwidth}
          \adjustbox{width = \linewidth,trim={.40\width} {.40\height} {.40\width} {.45\height},clip}
          {\includegraphics{figs/synoct_org.pdf}}
          \caption{Zoomed of (a)}
          \label{fig-synoct_images-e}
      \end{subfigure} &
      \begin{subfigure}{0.24\textwidth}
          \adjustbox{width = \linewidth,trim={.40\width} {.40\height} {.40\width} {.45\height},clip}
          {\includegraphics{figs/synoct_noisy.pdf}}
          \caption{Zoomed of (b)}
          \label{fig-synoct_images-f}
      \end{subfigure} &
      \begin{subfigure}{0.24\textwidth}
          \adjustbox{width = \linewidth,trim={.40\width} {.40\height} {.40\width} {.45\height},clip}
          {\includegraphics{figs/synoct_chen.pdf}}
          \caption{Zoomed of (c)}
          \label{fig-synoct_images-g}
      \end{subfigure} &
      \begin{subfigure}{0.24\textwidth}
          \adjustbox{width = \linewidth,trim={.40\width} {.40\height} {.40\width} {.45\height},clip}
          {\includegraphics{figs/synoct_our.pdf}}
          \caption{Zoomed of (d)}
          \label{fig-synoct_images-h}
      \end{subfigure}  
\end{tabular}    \vspace{-1em}
      \caption{Speckle denoising results for the synthetic OCT image by solving the EE model. First row: (a) Ground-truth OCT image $(300\times 300)$. (b) Noisy image with multiplicative speckle noise (variance 0.02). (c) Image denoised by HWC algorithm, PSNR = 27.09. (d)  Image denoised by proposed HALM algorithm, PSNR = 27.48. Second row:  Zoomed views of red boxes in (a), (b), (c), and (d) shown in (e), (f), (g), and (h) respectively.}
      \label{fig-synoct_images}
\end{figure}

In this experiment, we apply the HALM algorithm to the EE model to denoise the corrupted optical coherence tomography (OCT) image.
According to statistical optics, the noise in the OCT image is a multiplicative speckle noise.
To recover a high-quality OCT image, we first perform logarithmic compression on the degraded OCT image, which makes the noise in the converted image additive.
We denoise the transformed image by solving the EE model~\eqref{originalEEM} with the HALM algorithm.
Finally, we obtain the restored OCT image after exponential transformation.

We experiment with a synthetic OCT image generated by the Gaussian function provided in \cite{He2021Penalty}.
The variance of the speckle noise is 0.02.
The  HALM algorithm is compared with the HWC algorithm \cite{He2021Penalty}.
\Cref{fig-synoct_images}(a) is the ground-truth OCT image, \cref{fig-synoct_images}(b) is the noisy image, and \cref{fig-synoct_images}(c) and (d) are denoised images by the HWC  and HALM algorithm for the EE model, respectively.
The HALM algorithm performs well at removing the noise and preserving features for the OCT image inferred from \Cref{fig-synoct_images}.
It renders better PSNR values than the HWC algorithm, since the PSNR value of the denoised image using the  HALM algorithm is 27.48, and that using the  HWC algorithm is 27.09.

\begin{figure}[]
\centering
\begin{tabular}{@{}c@{\hskip 3ex}c@{}}
\begin{subfigure}{0.3\textwidth}
    \includegraphics[width = \linewidth]{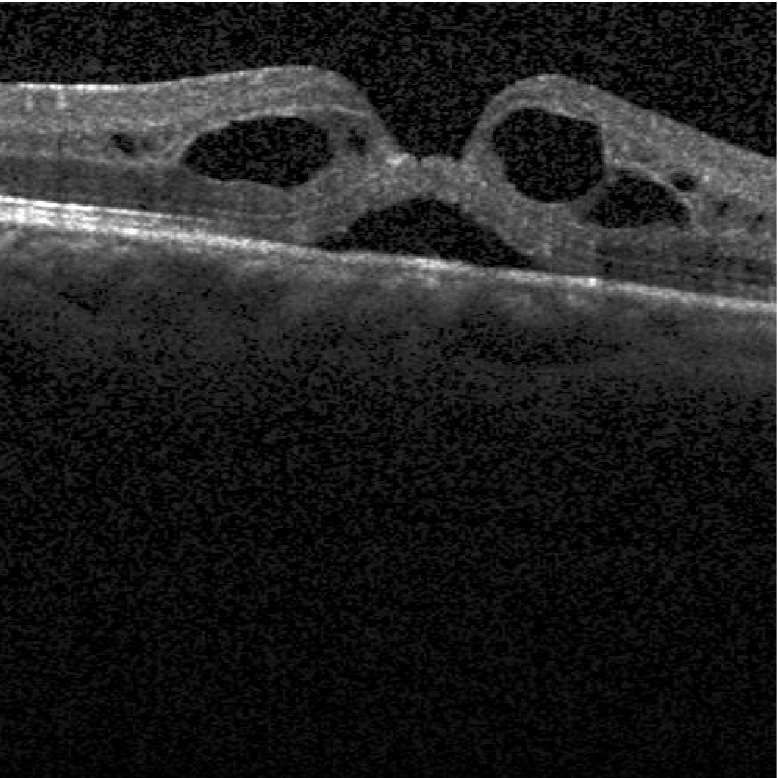}
\end{subfigure} &
\begin{subfigure}{0.3\textwidth}
    \includegraphics[width = \linewidth]{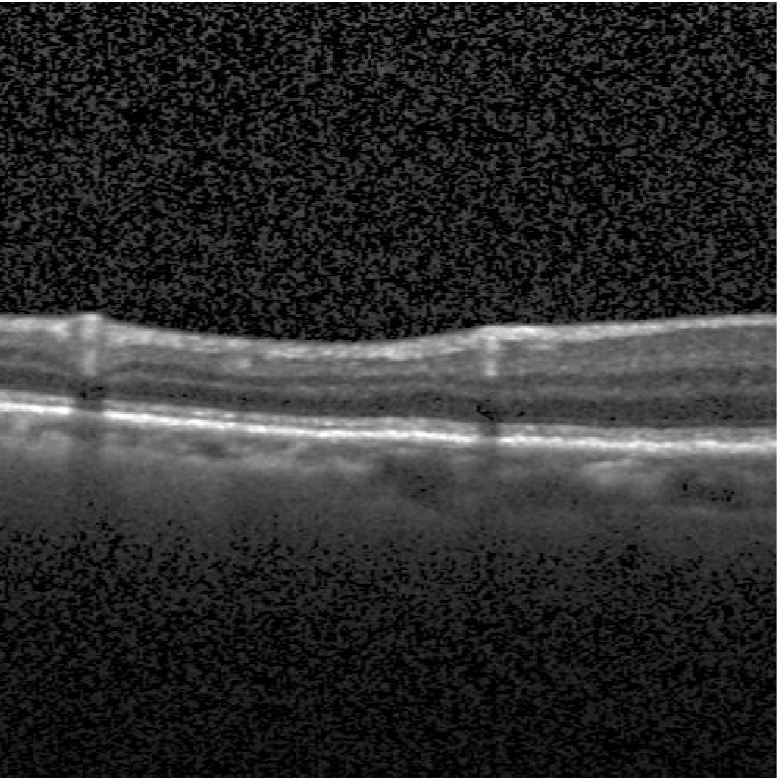}
\end{subfigure} 
\\ [2ex]
\begin{subfigure}{0.3\textwidth}
    \includegraphics[width = \linewidth]{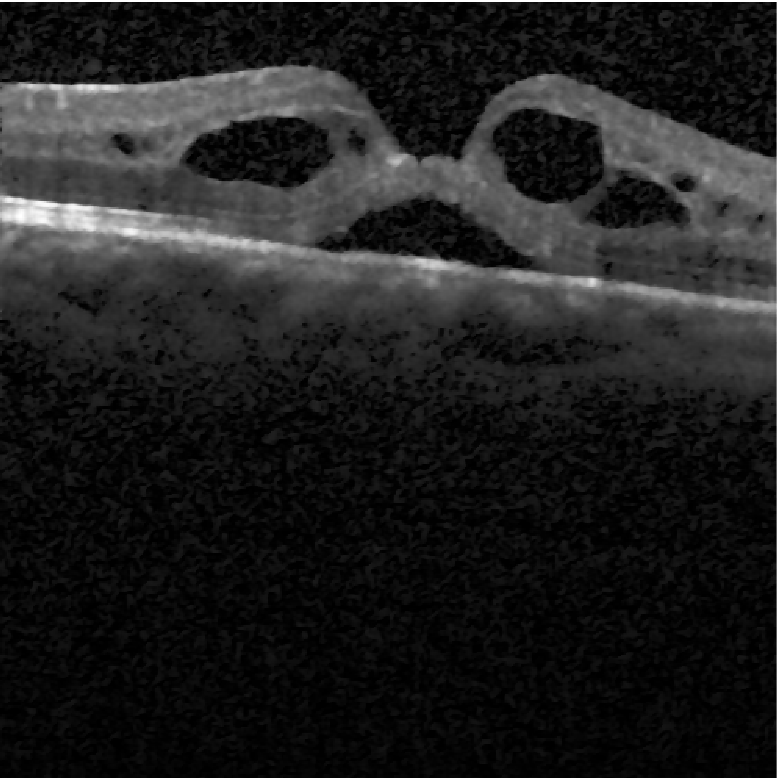}
\end{subfigure} &
\begin{subfigure}{0.3\textwidth}
    \includegraphics[width = \linewidth]{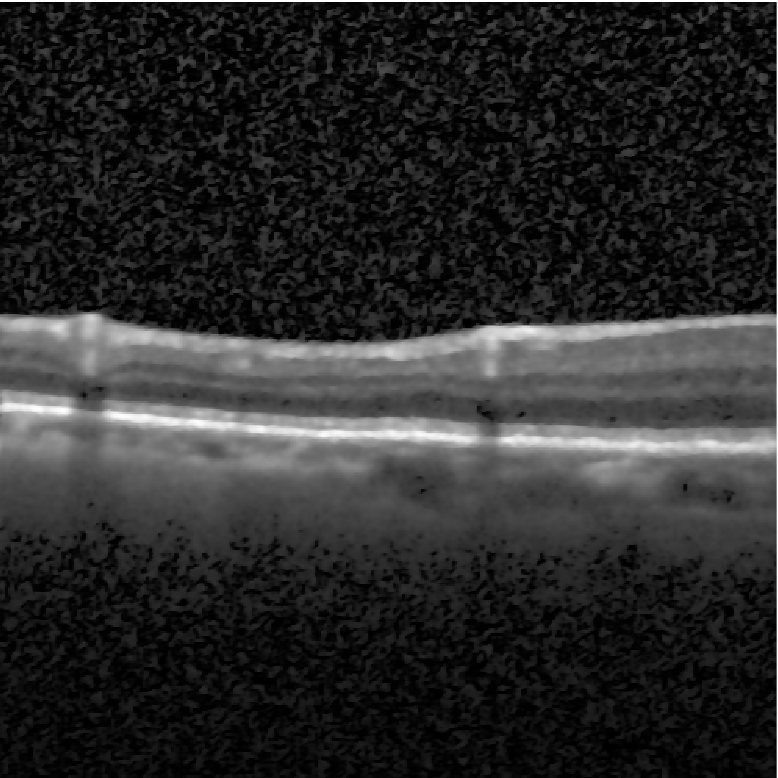}
\end{subfigure}
\end{tabular}
\caption{Speckle denoising results for real OCT images by solving the EE
model. The first row shows two noisy OCT images $(496\times 496)$ in the large dataset of labeled OCT and chest X-Ray images \cite{Kermany2018Cell}. 
The second row shows the corresponding denoised images by the proposed HALM algorithm.}
\label{fig-Real-OCT-denoising}
\end{figure}

We also test the HALM algorithm on denoising real OCT images from the large dataset of labeled OCT and chest X-ray images \cite{Kermany2018Cell}. Two representative OCT images are selected from the dataset. The results are shown in \Cref{fig-Real-OCT-denoising}. The noisy images are displayed in the first row, and the corresponding denoised results using the HALM algorithm are shown in the second row. As illustrated in \Cref{fig-Real-OCT-denoising}, the HALM algorithm effectively reduces  noise in the regions of interest while preserving important features and anatomical structures in the real OCT images.

\subsection{Gaussian denoising by the TRV model}

\begin{figure}[]
\centering
\begin{tabular}{@{}c@{\hskip 0.5ex}c@{\hskip 0.5ex}c@{\hskip 0.5ex}c@{}}
\begin{subfigure}{0.24\textwidth}
    \includegraphics[width = \linewidth]{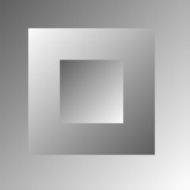}
    \caption{Ground-truth}
\end{subfigure} &
\begin{subfigure}{0.24\textwidth}
    \includegraphics[width = \linewidth]{figs/shading_TRV_noisy.pdf}
    \caption{Noisy}
\end{subfigure} &
\begin{subfigure}{0.24\textwidth}
    \includegraphics[width = \linewidth]{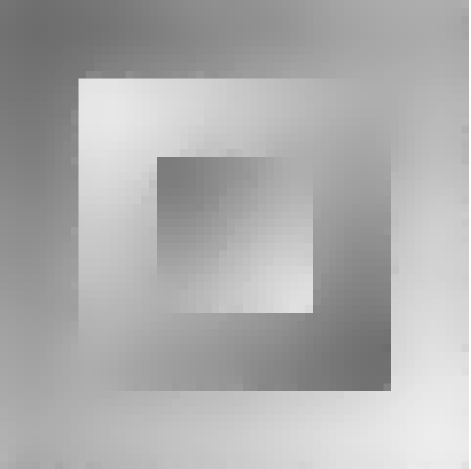}
    \caption{DTC}
\end{subfigure} &
\begin{subfigure}{0.24\textwidth}
    \includegraphics[width = \linewidth]{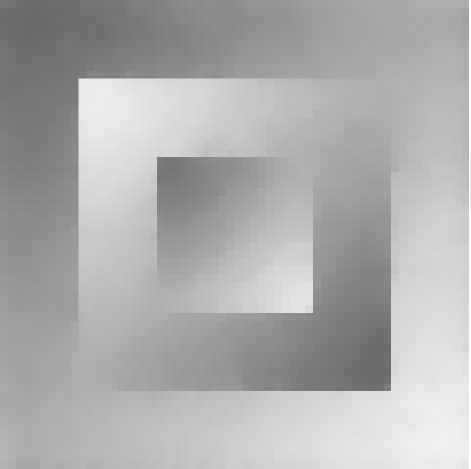}
    \caption{HALM-TRV}
\end{subfigure}
\end{tabular} \vspace{-1em}
\caption{Gaussian denoising results by solving the TRV model. (a) Ground-truth ``shading" image $(60\times 60)$. (b) Noisy image with additive Gaussian noise (mean 0, variance 0.0015). (c) Image denoised by DTC algorithm with PSNR/SSIM values of 32.08/0.9356. (d) Image denoised by HALM-TRV algorithm with PSNR/SSIM values 33.67/0.9410.}
\label{fig-TRV-denoising}
\end{figure}

In this experiment, we use the HALM algorithm described in \cref{sec-extension} to solve the TRV model, which is called the HALM-TRV algorithm, and compare it with the discrete total curvature (DTC) algorithm proposed in \cite{Duan_Min_discrete_total_curva}. The DTC method transfers the TRV model to a re-weighted total variation minimization problem and solves it by the alternating direction method of multipliers (ADMM) based algorithm. The source MATLAB code of the DTC method is kindly provided by the authors of \cite{Duan_Min_discrete_total_curva}. We experiment with the ``shading" image and degrade it with the additive Gaussian noise with a mean of 0 and a variance of 0.0015 (see \cref{fig-TRV-denoising} (a) and (b)). We set $a = 0.015$ and $b=0.005$ in the TRV model. For the HALM-TRV algorithm, we choose $\alpha = 4$ and $\tau=0.5$. For the DTC algorithm, the   parameters  are recommended by the authors of \cite{Duan_Min_discrete_total_curva}. \Cref{fig-TRV-denoising} (c) and (d) show the denoised images by the HALM-TRV and DTC algorithms, respectively. The DTC algorithm generates a restored image with a PSNR value of 32.08 and an SSIM value of 0.9356, and the HALM-TRV algorithm produces a denoised image with a PSNR value of 33.67 and an SSIM value of 0.9410. We also record the CPU running time of the two algorithms. The CPU time of the DTC algorithm is 1.2900 seconds, and that of the HALM-TRV algorithm is 0.2764 seconds. We notice that the proposed HALM-TRV algorithm outperforms the DTC algorithm in the quality of the denoised image and the computational cost for solving the TRV model. 

\section{Conclusions}\label{sec-conclusion}

We have proposed a novel bilinear decomposition for the EE model and developed a fast  HALM algorithm for its discrete form. We rigorously prove the convergence of the generated minimizing sequence from the algorithm and then validate it numerically. A host of  numerical experiments are conducted to demonstrate the performance of the  new HALM algorithm in its computation cost and the quality of the recovered images. The proposed algorithm has a great potential for general curvature-based  models. 

In reference to the latest operator-splitting algorithm presented in \cite{Liu2022Gaussian}, we need to point out that the reformulation strategy and the solution method for each subproblem devised for the HALM algorithm are quite different from the strategy developed in \cite{Liu2022Gaussian}. It seems the HALM algorithm cannot be used directly for the Gaussian curvature model to gain its computational efficiency since it will not lead to easy-to-solve sub-minimization problems and thereby lose its computational advantage. Nontrivial splitting technique should be developed.
The convergence of the HALM algorithm for the more general models has not yet been established when the smoothness of the objective function is lacking (e.g. TAC model and the Gaussian curvature regularization model \cite{Liu2022Gaussian}). They will be a part of our future work.

\section*{Acknowledgments}

We would like to thank Drs. Liangjian Deng, Fang He and Qiuxiang Zhong for providing the source codes of the DGT, HWC, and DTC methods, respectively, and for their fruitful discussions about the experiments.


\bibliographystyle{siamplain}
\bibliography{ref}

\end{document}